\documentclass[11pt]{amsart}
\usepackage{amsmath, amscd, eucal, amsthm, rotating}

\usepackage{pdfsync}

\usepackage{bbm}
\usepackage{times}
\usepackage[T1]{fontenc}
\usepackage{mathrsfs}
\usepackage{latexsym}

\input xy 
\xyoption{all}

\newif\ifPDF
\ifx\pdfoutput\undefined 
        \ifx\pdfversion\undefined
                \PDFfalse
        \fi
\else 
        \ifnum\pdfoutput > 0
                \PDFtrue
        \else
                \PDFfalse
        \fi
\fi

\numberwithin{equation}{section}

\newtheorem{proposition}[equation]{Proposition}
\newtheorem{theorem}[equation]{Theorem}		
\newtheorem{corollary}[equation]{Corollary}
\newtheorem{lemma}[equation]{Lemma}
\newtheorem{definition}[equation]{Definition}
\newtheorem{remark}[equation]{Remark}
\newtheorem{example}[equation]{Example}

\newcommand{\real}{\mathop{\mathfrak{Re}}\nolimits}
\newcommand{\imag}{\mathop{\mathfrak{Im}}\nolimits}

\newcommand{\C}{{\mathbb C}}
\newcommand{\R}{{\mathbb R}}
\newcommand{\Z}{{\mathbb Z}}

\newcommand{\Ocal}{{\mathcal O}}

\newcommand{\B}{{\mathfrak B}}

\newcommand{\Hilbzero}{\Omega^0_\R(M,L)}
\newcommand{\Hilbone}{\Omega^{0,1}_\R(M,L)}

\newcommand{\A}{{\mathcal A}}
\newcommand{\poisson}{{\mathcal P}}
\newcommand{\N}{{\mathcal N}}

\newcommand{\Rz}{{\mathcal R}}

\newcommand{\coker}{\mathop{\rm coker}\nolimits}

\newcommand{\Ric}{\mathop{\rm Ric}\nolimits}

\newcommand{\area}{\mathop{\rm area}\nolimits}

\newcommand{\tr}{\mathop{\rm Tr}\nolimits}
\newcommand{\dbar}{\bar\partial}

\newcommand{\lra}{\longrightarrow}

\newcommand{\smallA}{\mbox{\fontsize{5}{8}\selectfont $A$}}
\newcommand{\smallAr}{\mbox{\fontsize{6}{9}\selectfont $Ar$}}

\newcommand{\phia}{\Phi^{\mbox{\fontsize{5}{8}\selectfont $A$}}}

\newcommand{\Rone}{R^{\mbox{\fontsize{3}{6}\selectfont $(1)$}}}
\newcommand{\Rtwo}{R^{\mbox{\fontsize{4}{6}\selectfont $(2)$}}}
\newcommand{\Ronetwo}{R^{\mbox{\fontsize{3}{6}\selectfont $(1,2)$}}}
\newcommand{\Ri}{R^{\mbox{\fontsize{3}{6}\selectfont $(i)$}}}

\newcommand{\one}{{\mathbbm 1}}
\newcommand{\onep}{{\mathbbm 1}_{ p}}

\newcommand{\smallone}{\mbox{\fontsize{3}{6}\selectfont $(1)$}}
\newcommand{\smalltwo}{\mbox{\fontsize{3}{6}\selectfont $(2)$}}

\newcommand{\smalla}{\mbox{\fontsize{6}{8}\selectfont $A$}}

\newcommand{\phiaone}{\Phi^{\smalla,\smallone}}
\newcommand{\phiatwo}{\Phi^{\smalla,\smalltwo}}

\newcommand{\dl}{D_L}
\newcommand{\dla}{D_L^{\smalla}}

\newcommand{\laplace}{{\square_L}}

\newcommand{\bm}{\mathop{{\bf b}_{\partial M}}\nolimits}
\newcommand{\bmp}{\mathop{{\bf b}_{\partial M}^{\prime}}\nolimits}
\newcommand{\bmpp}{\mathop{{\bf b}_{\partial M}^{\prime\prime}}\nolimits}
\newcommand{\bmgamma}{\mathop{{\bf b}_{\partial M_\Gamma}}\nolimits}
\newcommand{\bmgammap}{\mathop{{\bf b}_{\partial M_\Gamma}^{\prime}}\nolimits}
\newcommand{\bmgammapp}{\mathop{{\bf b}_{\partial M_\Gamma}^{\prime\prime}}\nolimits}
\newcommand{\bgamma}{\mathop{{\bf b}_\Gamma}\nolimits}
\newcommand{\bgammap}{\mathop{{\bf b}_\Gamma^{\prime}}\nolimits}
\newcommand{\bgammapp}{\mathop{{\bf b}_\Gamma^{\prime\prime}}\nolimits}

\newcommand{\Bm}{\B(\partial M, \imath^\ast L)}
\newcommand{\Bmp}{\B'(\partial M, \imath^\ast L)}
\newcommand{\Bmpp}{\B''(\partial M, \imath^\ast L)}
\newcommand{\Bmgamma}{\B(\partial M_\Gamma, \imath^\ast L)}

\newcommand{\Bgamma}{\B(\Gamma, \imath^\ast L)}

\newcommand{\Bgammapp}{\B''(\Gamma, \imath^\ast L)}

\newcommand{\Log}{\mathop{\rm Log}\nolimits}
\newcommand{\Det}{\mathop{\rm Det}\nolimits}

\newcommand{\alv}{\mathop{\rm alv}\nolimits}

\begin{document}


\title[Gluing formulas for Dolbeault laplacians]
{Gluing formulas for determinants of Dolbeault\\ laplacians on Riemann surfaces}

\author[R.A. Wentworth]{Richard A. Wentworth}

\address{
Department of Mathematics,
   University of Maryland,
   College Park, MD 20742}

\thanks{Research supported in part by NSF grant DMS-1037094.}

\email{raw@umd.edu}

\date{\today}

\begin{abstract}
We present gluing formulas for zeta regularized determinants of Dolbeault laplacians on Riemann surfaces.  These are expressed in terms of determinants of associated operators on surfaces with boundary satisfying  local elliptic boundary conditions.  The conditions are defined using the additional structure of a framing, or trivialization of the bundle near the boundary.  An application to the computation of bosonization constants follows directly from these formulas.
 \end{abstract}



\maketitle

\thispagestyle{empty}


\baselineskip=16pt

\bigskip
\begin{center}
\emph{Dedicated to Professor Peter Li on the occasion of his 60th birthday.}
\end{center}

\section{Introduction}

Given a conformal metric $\rho$  on a closed Riemann surface $M$ 
of genus $g$ and  a hermitian metric 
$h$  on a holomorphic line bundle $L\to M$, 
let $\laplace=2\dbar_L^\ast\dbar_L$ be the Dolbeault laplacian acting on sections of $L$.
Determinants $\Det\laplace$ are  defined as the zeta 
regularized product of eigenvalues
and are functions of $\rho$, $h$, and the moduli of $M$ and $L$
 (see Section \ref{sec:determinant}; in the case of a kernel the notation
$\Det^\ast\laplace$ is used to emphasize that the zeta function is
defined using only nonzero eigenvalues).  In this paper, we derive gluing formulas for $\Det^\ast\laplace$ when $M$ is cut along closed curves, generalizing analogous identities obtained in \cite{Fo, BFK}.

The main difference with the scalar case is an appropriate choice of boundary 
conditions for $\laplace$ on a surface with boundary. 
 Local complex linear boundary conditions for the $\dbar$-complex
 do not exist, and it is common instead to impose spectral boundary conditions
 (for gluing formulas in this case, see for example
\cite{Park} and the references therein).
 By contrast,
in this paper we introduce local elliptic boundary conditions 
for sections of
 holomorphic bundles  equipped with a \emph{framing},  by which we mean
 a choice of trivialization near the boundary. 
These \emph{Alvarez boundary conditions} are of mixed Dirichlet--Robin type and come
 from the  splitting  of sections $\Phi$ near the
boundary into real 
and imaginary parts (denoted $\Phi'$ and $\Phi''$) made possible by the framing. 
The conditions are similar to  those studied  by Alvarez in \cite{A} for the case of traceless symmetric tensors -- hence, the name -- where there is a canonical choice of framing  (see also \cite{Gilkey03,BG}).
Because of the asymmetry, the boundary conditions 
 are manifestly not 
complex linear.  In particular, the $\dbar$-operator and  the Dolbeault laplacian must be regarded as
real operators $P_L$ and  $D_L$, respectively (see Section \ref{sec:real}).  
The Alvarez conditions on
 a  section $\Phi$ are then $(\Phi'', (P_L\Phi)'')\bigr|_{\partial M}=0$.
The advantage, however, is that
 the boundary value problem is  compatible with a similar BVP on 
the adjoint bundle. 
This leads to an index theorem for $P_L$ and
a generalization of the Polyakov-Alvarez formula for $\Det^\ast \dl$, which measures the variation under 
  conformal changes of $(\rho, h)$ (see Theorems \ref{thm:index} and \ref{thm:liouville}). Moreover, on a closed surface, $\Det^\ast \dl= \left(\Det^\ast\laplace\right)^2$, so a gluing formula for $\dl$ provides one for $\laplace$ as well.

To state the main result, let $\Gamma\subset M$ be a collection of disjoint embedded oriented closed curves, and let $M_\Gamma$ denote the manifold with boundary obtained by cutting $M$ along $\Gamma$. Then a line bundle $L\to M$ pulls back to $M_\Gamma$ (we use the same notation $L$).
There is a difference map $\delta_\Gamma$ on sections over $M_\Gamma$ which measures the difference of boundary values on each of the  two components of $\partial M_\Gamma$ covering a component $\Gamma$. With this notation, we have the following.

\begin{theorem} \label{thm:main}
Given a framing  of $L$ near $\Gamma$, 
let $\{\Phi_i\}$ $($resp.\ $\{ \phia_i\}$$)$
 be a basis for $\ker\dl$ on $M$ $($resp.\ for $\ker \dla$  on $M_\Gamma$ with Alvarez boundary conditions$)$.  
 Let $\det(\Phi_i,\Phi_j)$ denote the determinant in $(i,j)$ of the $L^2$-inner product on sections over $M$.  Similarly for the sections $\phia_i$ on $M_\Gamma$.  Also $\det(\Phi_i,\Phi_j)_\Gamma$ denotes the determinant in $(i,j)$ of the $L^2$-inner product of restrictions of sections to $\Gamma$, and similarly for $\delta_\Gamma\phia_i$.
 Assume the framing is generic
 in the sense of Definition \ref{def:generic_framing}. Then for any choice of $Q$, a self-adjoint elliptic positive pseudo-differential operator of order one on $\Gamma$, we have
$$
\left[\frac{{\Det}^\ast\dl}{\det(\Phi_i, \Phi_j)}\right]_M= c_Q\,  \left[\frac{{\Det}^\ast\dla}{\det(\phia_i,\phia_j)}\right]_{M_\Gamma}
\frac{\det(\delta_\Gamma\phia_i,\delta_\Gamma\phia_j)_\Gamma }{\det(\Phi_i'',\Phi_j'')_\Gamma}\, {\Det}_Q^\ast \N_\Gamma
$$
where $c_Q=2^{-\zeta_Q(0)}$ and $\N_\Gamma$ is a Neumann jump operator acting on sections over $\Gamma$ associated to the boundary value problem on $M_\Gamma$.
See Section \ref{sec:factorization} for more details.
\end{theorem}

An important point is that because of the mixed boundary conditions,  $\N_\Gamma$ 
 is here an
 operator of order zero, rather than of order
 one as in the scalar case \cite{Fo,BFK}.  Following Friedlander-Guillemin \cite{FG},
we define its  determinant  by choosing  a  regularizer $Q$ (see the  definition \eqref{eqn:def_det}).  
Actually, in this case the dependence on the choice $Q$ is simply an overall  constant $c_Q$.

In Section \ref{sec:asy_neumann}, we study the asymptotics of the 
determinant of the Neumann jump operator as $\Gamma$ shrinks to a point.  For the scalar case, this was a key
 step in \cite{W}.  In a similar manner we find that the Neumann jump
 operator takes a standard form in the limit and that the asymptotic behavior of its determinant may be determined explicitly.  As an application, the gluing formula can be used to cut and paste determinants for line bundles of different degrees.
 In particular,
 we give a new proof of a result on the behavior of determinants on exact sequences 
 $$
 0\lra L \lra L(p) \lra L(p)\bigr|_{\{p\}}\lra 0
 $$
 when the line bundles are equipped with \emph{admissible} metrics and $M$ with the \emph{Arakelov} metric (see Section \ref{sec:arakelov} for the definitions).

\begin{theorem}[Insertion Theorem, \cite{D,ABMNV, BismutLebeau91}] \label{thm:insertion}
Suppose $h^1(L)=\{0\}$, choose $p\in M$, and let $\Ocal(p)$ be the line bundle determined by the divisor $\{p\}$.  Let $\onep$ be a nonzero holomorphic section of $\Ocal(p)$ vanishing at $p$, and let $\hat\omega_0$ be a section of $L(p)=L\otimes\Ocal(p)$ that is nonvanishing at $p$.  Let $\{\omega_i\}_{i=1}^m$ be a basis of $H^0(L)$, and set $\hat \omega_i=\omega_i\otimes \onep$,
so that $\{ \hat\omega_i\}_{i=0}^m$ is a basis for $H^0(L(p))$. Fix admissible metrics on $L$ and $\Ocal(p)$ and the Arakelov metric on $M$, and let $L(p)$ have the induced metric.  Then
$$
2\pi \Vert \hat\omega_0(p)\Vert^2\, \frac{\Det^\ast \square_{L(p)}}{\det\langle \hat\omega_i , \hat\omega_j\rangle}
=\frac{\Det^\ast \laplace}{\det\langle\omega_i , \omega_j\rangle}
$$
where
$\langle \cdot, \cdot\rangle$ denotes the (hermitian) $L^2$-inner products on sections of $L$ and $L(p)$.
\end{theorem}
The equality above was  proven in \cite{D}, and also in \cite{ABMNV} (up to an overall constant) using the families index theorem. A higher dimensional version is proven in \cite{BismutLebeau91}.
This formula is a key step in the proof of the  \emph{bosonization formulas} 
on Riemann surfaces which relate zeta-regularized determinants of 
Laplace operators acting on sections of line bundles to 
determinants of scalar laplacians (see \cite{AMV, BN, VV,ABMNV, Bost, D, Sn, DS, F2}, 
and for their role in string theory \cite{DP3}). They are tantamount to a relationship between
the metrics  defined by Quillen and Faltings on the determinant of cohomology \cite{Q, Fg}. 
For the definition of an admissible hermitian metric and of the Arakelov metric and Green's function $G(z,w)$ used below, see Section \ref{sec:arakelov}. Then in the notation of 
\cite[Thms.\ 5.9 and 5.11]{F2}, the result states that  for $d\geq g-1$  and $M$ equipped with the Arakelov metric and associated Laplace-Beltrami operator $\Delta_M$,
 there are  constants $c_g$ and $\delta_g$ depending only on the genus, and $\varepsilon_{g,d}$ depending only on the genus and degree (normalized so that $\varepsilon_{g,g-1}=1)$,  such that
  for any holomorphic line bundle $L$ of degree $d$
with admissible metric $h$ and associated divisor $[L]$ satisfying $h^1(L)=0$,
  \begin{equation} \label{eqn:bose}
\frac{\Det^\ast\laplace}{\det \langle \omega_i,\omega_j\rangle}=\varepsilon_{g,d}\delta_g \exp(c_g/12)\left(\frac{\Det^\ast\Delta_M}{\area(M)\det\imag\Omega}
\right)^{-1/2} \frac{  \prod_{i\neq j}
G(p_i,p_j)}{\Vert\det\omega_i(p_j)\Vert^2}\Vert
\vartheta\Vert^2\bigl([L]-\sum_{i=1}^m p_i-\delta,\Omega\bigr) 
\end{equation}
where
  $m=d-g+1$, $\{p_i\}_{i=1}^m$ are generic points of $M$,   $\{\omega_i\}_{i=1}^m$ is any basis for $H^0(M, L)$, and the pointwise and $L^2$-metrics are taken with respect to $h$. Here, $\Omega$ is the period matrix for a choice of homology basis, $\vartheta(Z,\Omega)$ the theta function, and $\delta$ the Riemann divisor.  We refer to \cite{F2} for the origin of these constants, and in particular the distinction between $c_g$ and $\delta_g$.
  The unknown constants appearing in \eqref{eqn:bose} have been determined by Gillet-Soul\'e \cite{GSo, So} and, using different methods, by J. Jorgenson \cite{J2} and in \cite{W}. 
  For example, the result of \cite[Theorem 1.3]{W} is
  $$c_g=-8\log(2\pi)+(g-1)(24\zeta'(-1) - 1 -2\log\pi)$$
where $\zeta(s)$ is the Riemann zeta function.
   The remaining  values follow from Theorem \ref{thm:insertion}.  For completeness, we record the full result here.
The following is a generalization of the genus 1 computation in \cite[p.\ 117]{F2}.
 \begin{corollary} \label{cor:fay}
 Fay's constants $\delta_g$ and $\varepsilon_{g,d}$ defined in \cite[Thms.\ 5.9 and 5.11]{F2} have values
  \begin{align*}
  \delta_g&=(2\pi)^{g+1}\exp(c_g/6)
  \\
  \varepsilon_{g,d}&=(2\pi)^{g-1-d}
  \end{align*}
 \end{corollary}

\medskip
\noindent \emph{Acknowledgments.}  The author wishes to thank E. Falbel, A. Kokotov,
D.H. Phong, and the  referees for their suggestions. He especially thanks
 S. Zelditch for many discussions and in particular for pointing out ref.\ \cite{FG}. 
  The hospitality of
 the University of Paris 6 and the IHES, where a portion of
 this work was completed, is also gratefully acknowledged.

\section{The mixed boundary value problem}  \label{S:bvp}

 \subsection{Real structures}  \label{sec:real}
  We begin with a construction that is completely elementary
 but will nevertheless serve to make precise the notions of a real operator and a real structure used in this paper.
  Let $V$ be a complex Hilbert space with hermitian 
inner product $\langle\cdot,\cdot\rangle$ and dual
 space $V^\ast$.  Let $\Rz: V^\ast\to V$ be the complex antilinear
 isomorphism given by the Riesz representation: 
 $f(a)=\langle a, \Rz(f)\rangle$, for all $a\in V,  f\in V^\ast$.
 Note that the complex antilinear involution 
 $$
 \imath : V\oplus V^\ast\longrightarrow V\oplus V^\ast 
: (a,f)\mapsto (\Rz(f), \Rz^{-1}(a))
 $$
 satisfies
 $
 \langle \imath(a_1, f_1), \imath(a_2, f_2)\rangle 
=\overline {\langle (a_1, f_1), (a_2, f_2)\rangle} 
 $
 for the induced inner product on $V\oplus V^\ast$.  Define
 \begin{equation} \label{eqn:vr}
 V_\R=\text{Fix}(\imath)=\left\{ (a, \Rz^{-1}(a)) : a\in V\right\}
 \end{equation}
 The map 
 $
 \jmath : V\to V_\R : a\mapsto A=(a, \Rz^{-1}(a))
 $
 is then an $\R$-linear isomorphism.  The real vector space $V_\R$ inherits a complete
  inner product $(\cdot, \cdot )$ from 
$V\oplus V^\ast$, and
 \begin{equation} \label{eqn:dot1}
 (\jmath a_1, \jmath a_2)=2\real\langle a_1, a_2\rangle
 \end{equation}
 
 Let $T:V\to W$ be a (possibly unbounded) linear operator between complex Hilbert spaces. Then 
 $\Rz^{-1}T\Rz : V^\ast\to W^\ast$
  is also linear (with domain $\Rz^{-1}({\rm Dom}\, T)$).
  The associated operator
  $
  (T, \Rz^{-1}T\Rz) : V\oplus V^\ast\to W\oplus W^\ast
  $
  commutes with the involution $\imath$ and hence induces a real linear map $P_T: V_\R\to W_\R$ that makes the following diagram commute.
  \begin{equation*}
\xymatrix{
V  \ar[r]^{\jmath} \ar[d]_{T} & V_\R  \ar[d]^{P_T} \\
W\ar[r]^{\jmath} &  W_\R}
\end{equation*}
We call $P_T$ the {\bf real operator} associated to $T$.
Note that in the case $W=V$, it follows that the spectrum of 
$P_T: V_\R\to V_\R$ coincides with the real spectrum of $T:V\to V$ with 
\emph{twice} the multiplicity:  if $a\in V$ is nonzero with 
$Ta=\lambda a$ and $\lambda\in\R$, then $\jmath a$ and $\jmath(i a)$
 are independent eigenvectors of $P_T$, both  with eigenvalue $\lambda$.
 
 Finally, suppose that $V$ has a {\bf real structure}.  By this we mean a complex antilinear involution $\sigma : V\to V$ satisfying
 \begin{equation} \label{eqn:dot2}
 \langle \sigma a_1, \sigma a_2\rangle =\overline{\langle a_1, a_2\rangle}
 \end{equation}
 Then $\sigma_\R=\jmath\circ\sigma\circ\jmath^{-1}$ gives an 
involution of $V_\R$ which, by \eqref{eqn:dot1} and \eqref{eqn:dot2}, is 
an isometry.  Let $V_\R'$, $V_\R''$ denote the $+1$, $-1$ 
eigenspaces of $\sigma_\R$, respectively.  Then we have an orthogonal 
decomposition $V_\R=V_\R'\oplus V_\R''$.  For $A\in V_\R$, $A=A'+A''$, 
where $A'=(1/2)(A+\sigma_\R A)$, $A''=(1/2)(A-\sigma_\R A)$.  
We refer to $A'$ and $A''$ as 
the \emph{real} and \emph{imaginary} parts of $A$.
 There is a natural almost complex structure $J$ on $V_\R$ 
given by $JA=\jmath (i\jmath^{-1} (A))$.  A calculation shows that
 $(JA_1, JA_2)=(A_1, A_2)$, and
  $J(V_\R')\subset V_\R''$, $J(V_\R'')\subset V_\R'$.
 As a consequence, if we define a symplectic structure on $V_\R$ by the pairing
$(A_1, JA_2)$, then $V_\R'$ and $V_\R''$ are lagrangian subspaces 
(i.e.\ maximal isotropic).

 \subsection{Framed boundary conditions} \label{sec:abc}
  We apply the construction of Section \ref{sec:real} to sections of
 hermitian holomorphic line bundles on $M$. 
  Let $M$ be a compact Riemann surface of genus $g$ with a (non-empty) boundary
 $\partial M$ and inclusion $\imath:\partial M\hookrightarrow M$.
 Without loss of generality, we may assume that $M$ is obtained from a closed Riemann surface by deleting finitely many 
 disjoint coordinate disks.  
Each component of $\partial M$ has an open neighborhood  in $M$ biholomorphic to an annulus
 $\{ r_1\leq |z|< r_2\}$.  We will refer to such a $z$ as an \emph{annular coordinate}.

Let $L\to M$ be a holomorphic line bundle.  
 A  holomorphic structure on $L$ is equivalent to a Dolbeault operator 
$\dbar_L:\Omega^0(M,L)\to \Omega^{0,1}(M,L)$
 satisfying the Leibniz rule. Equip $M$ with a conformal metric $\rho$ and $L$ with a hermitian metric $h$.
 The holomorphic and hermitian structures on $L$ give a unique unitary Chern connection $D=(\dbar_L,h)$, as well as an adjoint operator $\dbar_L^\ast$, and similarly on $L^\ast$.   We will use the standard notation
  $h^0(L)=\dim_\C\ker\dbar_L$,  $h^1(L)=\dim_\C\coker\dbar_L=\dim_\C\ker\dbar_L^\ast$.

    There is a natural hermitian inner product on the space $\Omega^0( M,
  L)$ of smooth sections of $L$
given by 
$$
\langle s_1, s_2\rangle_{M} =\int_{ M} dA_\rho\,  \langle s_1, s_2\rangle_h
$$
where $dA_\rho$ is the area form on $M$ coming from the metric $\rho$. 
The dual space  is given by integration on $M$:
 $
 \Omega^0(M, L)^\ast\simeq \Omega^{1,1}(M, L^\ast)
 $. 
 Then
 \begin{equation}  \label{eqn:omega0_real}
\Omega^0_{\R}(M,  L)\subset \Omega^0(M,  L)\oplus \Omega^{1,1}(M,  L^\ast)
 \end{equation}
  is the real vector space constructed as 
 in \eqref{eqn:vr}.  Strictly speaking, here we should work with the $L^2$ and Sobolev completions.  These are defined using the Chern connection $D$. Since this is standard, for notational simplicity we omit this from the notation.

 We can also carry out this construction on $(0,1)$-forms:
 \begin{equation} \label{eqn:omega1_real}
 \Omega^{0,1}_{\R}(M,  L)\subset \Omega^{0,1}(M,  L)\oplus \Omega^{1,0}(M,  L^\ast)
 \end{equation}
Denote the isomorphisms  of real vector spaces
\begin{align*} 
 \jmath_0: 
\Omega^0(M, L)&\longrightarrow \Omega^0_\R(M, L):
\varphi\mapsto \Phi \\
 \jmath_1: 
 \Omega^{0,1}(M, L)&\longrightarrow \Omega^{0,1}_\R(M, L): 
\psi\mapsto \Psi
\end{align*}
or simply by $\jmath$ when the meaning is clear.

 As in Section \ref{sec:real},  define a (real, unbounded) linear operator 
 $P_L:\Omega^0_\R(M,L)  \to \Omega^{0,1}_\R(M,L) $
  making the following diagram commute:
  \begin{equation*}
\xymatrix{
 \Omega^0(M,L)  \ar[r]^{\jmath_0} \ar[d]_{\dbar_L} &   \Omega^0_\R(M,L)  \ar[d]^{P_L} \\
 \Omega^{0,1}(M,L) \ar[r]^{\jmath_1} &   \Omega^{0,1}_\R(M,L) }
\end{equation*}
In terms of the decompositions \eqref{eqn:omega0_real} and \eqref{eqn:omega1_real}, it follows that
\begin{equation} \label{eqn:p}
P_L=\left(\begin{matrix} \dbar_L &0\\ 0& (\dbar_{L^\ast})^\ast\end{matrix}\right)
\end{equation}

  Now consider the boundary.  
   There is an hermitian inner product on $\Omega^0(\partial M,
 \imath^\ast L)$
given by 
$$
\langle s_1, s_2\rangle_{\partial M} =\int_{\partial M} ds_\rho\,  \langle s_1, s_2\rangle_h
$$
where $ds_\rho$ is the induced measure on $\partial M$.
Note that $\partial M$ inherits an orientation from $M$ and the outward normal.  Hence,
 integration gives an identification $\Omega^0(\partial M, \imath^\ast L)^\ast$ 
with $\Omega^1(\partial M, \imath^\ast (L^\ast))$.  With this understood, let 
\begin{equation} \label{eqn:v0}
\Omega^0_{\R}(\partial M, \imath^\ast L)\subset \Omega^0(\partial M, \imath^\ast
L)\oplus \Omega^1(\partial M, \imath^\ast (L^\ast))
\end{equation}
 be the real vector space constructed as 
 in the previous section.

 The
trace map
 \begin{equation}
 \Omega^0(M, L)\longrightarrow \Omega^0(\partial M, \imath^\ast L) : \varphi\mapsto \varphi\bigr|_{\partial M}
 \end{equation}
is induced by restriction.
 Using the Hodge star on $M$ to identify $\Omega^{1,1}(M, L^\ast)\simeq \Omega^{0}(M, L^\ast)$,
and on $\partial M$ to identify $\Omega^{1}(\partial M, \imath^\ast L^\ast)\simeq
\Omega^{0}(\partial M, \imath^\ast L^\ast)$,  there is a
similar restriction map
$$
\Omega^{1,1}(M, L^\ast)\simeq \Omega^{0}(M, L^\ast)\longrightarrow
\Omega^0(\partial M, \imath^\ast L^\ast)\simeq 
\Omega^{1}(\partial M, \imath^\ast L^\ast)
 $$
The restriction maps  combine to give a
 trace map 
$\Omega^0_\R(M, L)\to \Omega^0_{\R}(\partial M, \imath^\ast L)$.
We carry out the same construction with $\Omega^{0,1}(M, L)$.  Here, we define
$$
\Omega^1_{\R}(\partial M, \imath^\ast L)\subset \Omega^1(\partial M, \imath^\ast
L)\oplus \Omega^0(\partial M, \imath^\ast (L^\ast))
$$
In this case,  again using the Hodge star on $\partial M$
 the trace map $ \Omega^{0,1}_{\R}(M,  L)\to \Omega^1_{\R}(\partial M, \imath^\ast L)$ pulls-back the  forms and restricts the section.
  
 \begin{definition} \label{def:boundary_values}
Let
$$\Bm= \Omega^0_{\R}(\partial M, \imath^\ast L)\oplus \Omega^1_{\R}(\partial M,
\imath^\ast L)$$
 be the space of boundary data.   
 The trace map is the   (real) linear map:
$$
\bm :\Omega^0_\R(M, L)\longrightarrow \Bm :
 \Phi\mapsto  (\Phi, P_L\Phi)\bigr|_{\partial M}
$$
defined as above.  
\end{definition}

In order to define elliptic boundary conditions we will need  real structures.
These come from a choice of trivialization of $L$ 
near $\partial M$.

\begin{definition} \label{def:framing}
A {\bf framing} of a holomorphic line bundle $L\to M$ is a trivialization $($i.e.\ a nowhere vanishing holomorphic section$)$ $\tau_L$ of $L$ near $
\partial M$.
\end{definition}

An important example of a framing is the following
\begin{example} \label{ex:divisor}
Let $L$ be defined by a divisor  $D$  compactly supported in $M$.  Then by
construction $L$ has a
meromorphic section $\tau_L$  with zeros and poles
 exactly at $D$.  In particular, $\tau_L$ gives a
framing of $L$. While $\tau_L$ is only defined up to multiplication by a nonzero
constant, we shall refer to any such choice as a canonical framing. 
\end{example}

Given a framing 
and a section $\varphi$ of $L$ defined in a neighborhood of
$\partial M$, write $\varphi=(\varphi'+i\varphi'')\cdot\tau_L$, where
$\varphi'$, $\varphi''$ are real valued functions.  Then let
$\sigma(\varphi)=(\varphi'-i\varphi'')\cdot\tau_L$.  This defines a real
structure on $\Omega^0(\partial M, \imath^\ast L)$.
As in Section \ref{sec:real}, the boundary values of
$\Phi\in \Omega^0_\R( M,  L)$   therefore have real and imaginary parts $\Phi'$, $\Phi''$. 
The framing also gives a real structure on boundary values of  elements of $\Omega^{0,1}(M, L)$.  Indeed, there is natural 
isomorphism
 $T^{0,1}M \bigr|_{\partial M}\simeq T(\partial M)\otimes \C$.
   Equivalently, the Hodge star gives a $\C$-linear isomorphism
$\ast:\Omega^0(\partial M, \imath^\ast L)\simeq \Omega^1(\partial M, \imath^\ast
L)$ with $\ast^2=1$.  If $\sigma_0$ is the real structure on $\Omega^0(\partial
M, \imath^\ast L)$, then $\sigma_1=\ast\sigma_0\ast$ is a real structure on
$\Omega^1(\partial M, \imath^\ast L)$.
  We let $\Bmp$ (resp.\ $\Bmpp$) be the subspaces of $\Bm$ consisting of elements $(\Phi', \Psi')$ (resp. $(\Phi'', \Psi'')$).

\begin{itemize}
\item
Note that there is a natural pairing of $\Omega^0_{\R}(\partial M, \imath^\ast
L)$ and $\Omega^1_{\R}(\partial M, \imath^\ast L)$ defined as follows.  If $\Phi=\jmath_0(\varphi)$, $\Psi=\jmath_1(\psi)$ 
 then
\begin{equation} \label{eqn:pairing}
(\Phi,\Psi)_{\partial M}=2\real\int_{\partial M}  \langle \varphi,\psi\rangle_h
\end{equation}
\item The real structure defines an almost complex structure on
$\Omega^0_\R(\partial M, \imath^\ast L)$ and $\Omega^1_\R(\partial M, \imath^\ast L)$ as in Section \ref{sec:real}.  We extend this to an almost complex structure on the space of boundary values $\Bm$ by defining 
$$
J_{\partial M}=\left(\begin{matrix} 0 & \ast J\\ \ast J&0\end{matrix}\right)
$$
(for simplicity, we will denote this operator simply by $J$ as well).  This almost complex structure  and the pairing \eqref{eqn:pairing} give a symplectic structure on $\Bm$ defined by $(f, Jg)$.  As in Section \ref{sec:real}, the subspaces $\Bmp$ and $\Bmpp$ are then lagrangian.
\end{itemize}

   \begin{definition}
Let $\bmp$ and $\bmpp$ be the projections to the real and imaginary parts of $\bm$.
We call the equation $\bmp(\Phi)=0$ 
$($resp.\ $\bmpp(\Phi)=0$$)$ the \emph{real}
 $($resp.\ \emph{imaginary}$)$ \emph{Alvarez boundary conditions}.
\end{definition}

Note that ${\bf
b}'_{\partial M}$ and $\bmpp$ 
take values in lagrangian subspaces
of $\Bm$.
We will use the same notation for the boundary map on $ \Omega^{0,1}_{\R}(M,  L)$; namely,
$$
\bm :\Omega^{0,1}_\R(M, L)\longrightarrow \Bm :
 \Psi\mapsto  (P_L^\dagger\Psi, \Psi)\bigr|_{\partial M}
$$
where $P_L^\dagger$ is the formal adjoint of $P_L$.   Then $\bmp$ and $\bmpp$ are defined similarly.

 Since we  here assume that $\partial M\neq\emptyset$, by a theorem of Grauert $L$ admits a
global holomorphic trivialization $\one$ on $M$.  Then $\tau_L/\one$ is a
nowhere vanishing holomorphic function in a neighborhood of $\partial
M$. 
 We define the {\bf degree}  $\deg(\tau_L)$  of a framed line bundle to be the winding number of
$\tau_L/\one$  (with the outward normal, summed over all components of $\partial M$).  
Clearly, the definition of degree  is independent of
the choice of trivialization $\one$.
 Note the following two important examples.
\begin{example} \label{ex:degree}
\begin{enumerate}
\item Let $s$ be a meromorphic section of $L$
 satisfying imaginary Alvarez boundary conditions and with divisor $(s)$ compactly
 supported  in the interior of $M$.  Then
$\deg(\tau_L)=\deg(s)$.  
\item Let $L=K^q$, where the framing  is given by 
 $\tau_L=(-idz/z)^q$ in local annular coordinates near $\partial M$. Then 
$\deg(\tau_L)=-\chi(M)$. One can check that the real structure is independent of the choice of annular coordinate.
\end{enumerate}
\end{example}

The Alvarez boundary conditions are of mixed Dirichlet-Robin type.  
 Indeed, fix a framing $\tau_L$ of $L$, and let $h=\Vert \tau_L\Vert^2$.  Then on $\partial M$, define
\begin{equation}  \label{eqn:nu}
\nu_{L,h}=-\tfrac{1}{2}\partial_n\log h
\end{equation}
where $n$ is the \emph{outward} normal.
Also, let $\Pi_{\pm}=\frac{1}{2}(I\pm \sigma_\R)$ be the orthogonal projections to the real and imaginary parts. 
 Then it is easy to see that $\bmpp(\Phi)=0$ 
is equivalent to the conditions
\begin{align} \label{eqn:boundary_conditions}
\begin{split}
\Pi_- \Phi\bigr|_{\partial M}&=0 \\
(\nabla_n+S)\Pi_+\Phi \bigr|_{\partial M}&=0
\end{split}
\end{align}
where $n$ is the 
outward normal, $\nabla$ is the induced connection on the bundle of real sections, and
  $ S=\nu_{L,h}$.
Indeed, write $\varphi=(\varphi'+i\varphi'')\tau_L$.  
The Alvarez boundary conditions are $\varphi''=0$ and $\partial_n\varphi'=0$ on $\partial M$.
A local unitary frame is given by ${\bf e}_L=h^{-1/2}\tau_L$.  Since the connection form  in the frame ${\bf e}_L$ is purely imaginary, $ {\bf e}_L$ is parallel with respect to $\nabla$, and the result follows from the expression $\Pi_+\Phi=(\varphi' h^{1/2}){\bf e}_L$.

\subsection{Heat kernels and an index theorem}

 A straightforward calculation gives the following important integration by parts formula. For smooth sections $\Phi\in \Omega^0_\R(M,L)$ and $\Psi
\in \Omega^{0,1}_\R(M,L)$,
\begin{equation} \label{eqn:int_by_parts}
(P_L\Phi,\Psi)_M-(\Phi, P_L^\dagger \Psi)_M= \tfrac{1}{2}(\Phi, J\Psi )_{\partial M} 
\end{equation}
where  the pairing \eqref{eqn:pairing} appears on the right hand side. 
Define the laplacian $\dl=2P_L^\dagger P_L$  on smooth sections $\Omega^0_\R(M,L)$.  Then from \eqref{eqn:int_by_parts} we have

\begin{align} 
(\dl \Phi_1, \Phi_2)_M-(\Phi_1, \dl \Phi_2)_M
&=(\bm(\Phi_1), J\bm(\Phi_2)) \label{eqn:parts} \\
2(P_L\Phi_1 , P_L \Phi_2)-(\Phi_1, \dl\Phi_2)&=\left[ (\Phi_1'', J (P_L\Phi_2)')- ((P_L\Phi_2)'', J\Phi_1') \right]\label{eqn:positive} 
\end{align}

 Notice that the right hand sides of \eqref{eqn:parts} and  \eqref{eqn:positive} vanish identically
 for Alvarez boundary conditions. This gives positivity and formal self-adjointness of $\dl$.
   For the following result, see for example \cite[Lemma 1.11.1]{Gil}.

\begin{proposition}
Assuming either real or imaginary Alvarez boundary conditions, the formal adjoint $P_L^\dagger$ extends
to an unbounded operator on $\Hilbone$ as the 
 the $L^2$-adjoint of $P_L$  on $\Hilbzero$. Moreover, 
$\dl$  extends to an  unbounded self-adjoint non-negative elliptic operator $\dla$ on sections $\Hilbzero$ satisfying real (resp.\ imaginary) Alvarez boundary conditions.  A similar statement holds for the laplacian $2P_LP_L^
\dagger$ on $\Hilbone$.
\end{proposition}

We now make a choice: henceforth, unless otherwise indicated, by \emph{Alvarez boundary conditions} we will mean the condition $\bmpp(\Phi)=0$.  We  write $\dla$ when we wish to emphasize that the laplacian $\dl$ is acting on the space of sections satisfying Alvarez boundary conditions.

\begin{remark} \label{rem:kernels}
By \eqref{eqn:positive},  $\ker \dla\subset\ker P_L$.  Hence, $\ker\dla$ is real isomorphic to 
the space of holomorphic sections $\varphi$ of $L$ with  local
expression
$
\varphi=\varphi(z)\tau_L$ near $\partial M$,
satisfying
$
\imag(\varphi(z))\bigr|_{\partial M}=0
$.
\end{remark}

\begin{remark} \label{rem:adjoint}
If $\Phi$ is an eigensection of $\dl$ satisfying Alvarez boundary conditions with eigenvalue
$\lambda\neq 0$, then $P_L\Phi$
is an eigensection of $2P_LP_L^\dagger$ with the same eigenvalue
$\lambda$,  also satisfying Alvarez boundary conditions.
\end{remark}
\noindent
This  simple observation is the \emph{raison d'\^etre} of the
mixed boundary conditions we have chosen.  By contrast, if $\varphi$ is an eigensection of
$\laplace$ satisfying Dirichlet conditions, then $\dbar_L\varphi$ is a formal  eigensection of
$\dbar_L\dbar_L^\ast$, but does not necessarily 
satisfy an elliptic boundary condition.

We also note the following
\begin{proposition}[Serre duality] \label{prop:serre}
Fix a framing $\tau_L$ on  $L\to M$.  
Then with respect to the duality 
$$\Omega^{0,1}(M,L)\simeq (\Omega^0(K\otimes L^\ast))^\ast$$
the framing on $K\otimes L^\ast$ is induced by that on $L$ and $-idz/z$, where $z$ is an annular coordinate near $\partial M$.
In particular, with these Alvarez boundary conditions,
$
\coker P_L\simeq \ker P_L^\dagger\simeq (\ker P_{K\otimes L^\ast})^\dagger
$.
\end{proposition}

\begin{proof}
The usual proof of Serre duality applies, modulo the boundary conditions.  To understand these, choose a local annular coordinate $z$ near $\partial M$.  
Then with respect to the trivialization $\tau_L$, a smooth section  $\psi d\bar z\in \Omega^{0,1}(M,L)$ satisfies Alvarez boundary conditions if $\imag(i\psi e^{-i\theta})=0$ and $\imag(\dbar^\ast_L(\psi d\bar z))=0$ on $\partial M$.  The corresponding section of $\Omega^0(K\otimes L^\ast)$ is 
$\bar\psi h dz=iz\bar\psi h(-i dz/z)
$, and so the   Alvarez conditions  are $\imag(iz\bar\psi h)=0$ and  $\imag(\partial_{\bar z}(iz\bar\psi h)d\bar z)=0$ on $\partial M$.
But on $\partial M$, $\imag(i\psi e^{-i\theta})=0$ is equivalent to $\imag(iz\bar\psi h)=0$. 
In a similar way one shows $\imag(\partial_{\bar z}(iz\bar\psi h)d\bar z)=-h\imag(\dbar^\ast_L(\psi d\bar z))$.
 This proves the Proposition.
\end{proof}

In order to state a result for the small time expansion of the trace of the heat kernel, we will need the following quantities.
Let $\Omega_{L,h}$ denote the Hermitian-Einstein tensor (cf.\ \cite[IV.1.2]{Ko}).  In a local 
holomorphic frame we have
\begin{equation} \label{eqn:omega}
\Omega_{L,h}=i\ast F_{(\dbar_L,h)}=-\tfrac{1}{2}\Delta_\rho\log h
\end{equation}
 where $F_{(\dbar_L,h)}$ is the curvature of the Chern connection.
Note the following special case.
\begin{lemma} \label{lem:canonical}
Let $R_\rho$ and  $\kappa_\rho$ denote the scalar
 and geodesic curvatures of $M$ and  $\partial M$.
With the hermitian metric on $K$ induced from the metric on $M$, 
$
\Omega_{K,\rho^{-1}}=-(1/2)R_\rho
$. For the framing $-idz/z$, 
$
\nu_{{\mbox{\fontsize{6}{10}\selectfont $K$}},\rho^{-1}}=\kappa_{\rho}
$.
\end{lemma}

For the 
 short time expansion of heat kernels, 
we refer to \cite{BG} and \cite{Gil}.  In 
particular, we use the result in \cite[Sec.\ 5.3]{V} and the expression for $S$ in 
\eqref{eqn:boundary_conditions}.

\begin{proposition} \label{prop:heat}
Let $L\to M$ be a holomorphic line bundle on $M$ with framing $\tau_L$.  Let $\rho$ and $h$ be hermitian metrics on $M$ and $L$, respectively.
Then for any function $f$,  the trace with the heat kernel for the operator $\dla$ with Alvarez boundary conditions defined by $\tau_L$
 has the following short time expansion:
$$
\tr(fe^{-\varepsilon \dla})
=\frac{1}{2\pi\varepsilon}\int_M dA\, f   + \frac{1}{12\pi}\int_M dA\, f(6\Omega_{L,h}+R_\rho)+
\frac{1}{6\pi}\int_{\partial M} ds\, f(\kappa_\rho-3\nu_{L,h}) 
+O(\varepsilon^{1/2})
$$
\end{proposition}

\begin{theorem}[Index theorem] \label{thm:index}
Let $L\to M$ be a holomorphic line bundle on  $M$ with  framing $\tau_L$. 
 Then  for Alvarez boundary conditions,
\begin{equation} \label{eqn:index}
{\rm index}\, P_L=\dim_\R\ker P_L-\dim_\R \coker P_L=2\deg(\tau_L)+\chi(M)
\end{equation}
\end{theorem}

\begin{proof}
From Proposition \ref{prop:heat}, Lemma \ref{lem:canonical}, Remark \ref{rem:adjoint}, and Proposition \ref{prop:serre}
\begin{align*}
{\rm index}\, P_L &=\lim_{\varepsilon\to 0}\left\{ \tr(e^{-2\varepsilon P_L^\dagger P_L})-\tr(e^{-2\varepsilon P_L P_L^\dagger})\right\} \\
&=
\frac{1}{2\pi}\int_M dA\, (\Omega_{L,h}-\Omega_{KL^\ast,(\rho h)^{-1}})
-\frac{1}{2\pi}\int_{\partial M} ds\, (\nu_{L,h}-\nu_{KL^\ast,(\rho h)^{-1}}) \\
&= \frac{1}{2\pi}\int_M dA\, 2\Omega_{L,h}
-\frac{1}{2\pi}\int_{\partial M} ds\, 2\nu_{L,h} +
\frac{1}{4\pi}\int_M dA\, R_\rho
+\frac{1}{2\pi}\int_{\partial M} ds\, \kappa_\rho 
\end{align*}
By the Gauss-Bonnet Theorem, the last two terms give the Euler characteristic $\chi(M)$.
Write $\tau_L=f\one_L$, and let $h_0=\Vert\one_L\Vert^2$.  Then near $\partial M$, $h=|f|^2 h_0$,
and
$$
\deg(\tau_L)=\frac{1}{2\pi}\int_{\partial M}ds\,  \partial_n\log |f|
$$
 On the other hand,
\begin{align*}
\frac{1}{2\pi}\int_M dA\,
 \Omega_{L,h}-\frac{1}{2\pi}\int_{\partial M}\nu_{L,h}&
=-\frac{1}{4\pi}\int_{M} dA\, \Delta\log h_0+\frac{1}{4\pi}\int_{\partial M}ds\, \partial_n\log h \\
&=\frac{1}{4\pi}\int_{\partial M}ds\,  (-\partial_n\log h_0+\partial_n\log h) \\
&=\frac{1}{4\pi}\int_{\partial M}ds\,  \partial_n\log |f|^2 =\deg(\tau_L)
\end{align*}
The result follows.
\end{proof}

\begin{remark} \label{rem:alvarez_index}
By Example \ref{ex:degree}, if $K^q$ on $M$ is given the framing
$(-idz/z)^q$  for annular coordinates at each component of $\partial M$, then
$
\deg(\tau_{K^q})=-q\chi(M)
$.  Hence, by Theorem \ref{thm:index},
${\rm index}\,  P_{K^q}=(1-2q)\chi(M)
$.
This agrees with \cite[eq.\ (4.32)]{A}.
\end{remark}

 \subsection{Determinants of laplacians} \label{sec:determinant}
Following \cite{RS}, we define determinants as follows.
Suppose $M$ is closed with conformal metric $\rho$ and a hermitian holomorphic line bundle $L\to M$.
Let $\{\lambda_j\}_{j=1}^\infty$ be the spectrum of $\laplace$ and form the zeta function
$
\zeta_{\laplace}(s) =\sum_{\lambda_j>0} \lambda_j^{-s}
$.
Then  $\zeta_{\laplace}(s)$ converges for $\real(s)$ sufficiently large, and by a theorem of Seeley \cite{See} it is known that $\zeta_{\laplace}(s)$
is regular at $s=0$.   Then 
$
\log\Det^\ast \laplace := -\zeta_{\laplace}'(0)
$.
A similar definition applies to $\Det^\ast\dl$ on $M$, 
and to $\Det^\ast\dla$ when $M$ has boundary,
 $L$ has a framing, and we use Alvarez boundary conditions. 
 When it is understood that the spectrum is
 strictly positive, we will omit the asterisk and write $\Det\laplace$, etc.

When $M$ is closed,  $\dl$ acting on $\Omega_\R^0(M,L)$ is the same as $\laplace$ acting on $\Omega(M,L)$, regarded as a 
real operator (see Section \ref{sec:real}), and hence it has the same spectrum 
but with twice the multiplicity.
 Taking into account also the factor of $2$  in
 the definition of the real inner product  (see \eqref{eqn:dot1}), 
we have the following
\begin{lemma} \label{lem:real_complex}
If $M$ is a closed Riemann surface with line bundle $L\to M$.   
Then for all $\lambda>0$, 
$$\Det(\dl+\lambda)=[\Det(\laplace+\lambda)]^2$$
  Similarly, 
 $$\frac{\Det^\ast \dl}{\det(\Phi_i,\Phi_j)}=
\left(   2^{-h^0(L)} \frac{\Det^\ast \laplace}{\det\langle\omega_i,
\omega_j\rangle} \right)^2$$
 where $\{\omega_i\}_{i=1}^{h^0(L)}$ is a basis (over $\C$) 
for $H^0(M,L)$ and $\{\Phi_i\}_{i=1}^{2h^0(L)}$ is the associated basis (over $\R$) of $\ker \dl$ given by
\begin{equation} \label{eqn:assoc_basis}
\Phi_{2j}=\jmath(i\omega_j) \ ,\
\Phi_{2j-1}=\jmath(\omega_j) 
\end{equation}
for $j=1, \ldots, h^0(L)$.
\end{lemma}

The main result of this section is the following

\begin{theorem}[Polyakov-Alvarez formula] \label{thm:liouville}
Let $\{\Phi_i\}_{i=1}^m$, $\{\Psi_j\}_{j=1}^n$ be bases for $\ker P_L$ and $\ker P_L^\dagger$, respectively, with Alvarez boundary conditions.
Suppose the following relation for hermitian metrics: $ \rho=e^{2\sigma}\hat \rho$, $ h=e^{2f}\hat h$. Then
$$
\left[\frac{{\Det}^\ast\dla}{\det(\Phi_i,\Phi_j)\det(\Psi_i,\Psi_j)}\right]_{(\rho,h)}=
\left[\frac{{\Det}^\ast\dla}{\det(\Phi_i,\Phi_j)\det(\Psi_i,\Psi_j)}\right]_{(\hat \rho,\hat h)}\exp(S(\sigma, f))
$$
where
\begin{align} \label{eqn:liouville}
\begin{split}
S(\sigma, f)&=-\frac{1}{6\pi}\int_M dA_{\hat\rho}\left\{6\nabla f\cdot\nabla(\sigma+f)+|\nabla\sigma|^2   \right\} \\
&\qquad -\frac{1}{6\pi}\int_M dA_{\hat\rho}\left\{  6\Omega_{L,\hat h}(\sigma+2f)+R_{\hat\rho}(\sigma+3f) \right\} \\
&\qquad +\frac{1}{3\pi}\int_{\partial M} ds_{\hat\rho}\left\{ 3\nu_{L,\hat h}(\sigma+2f)-\kappa_{\hat\rho}(\sigma+3f)   \right\} 
\end{split}
\end{align}
\end{theorem}

\begin{proof}
The argument  follows \cite{A}; here we only sketch the ideas.  
  Let $\{\Phi_j\}$ be an orthonormal
 basis  of eigensections of $\dla$ with eigenvalues $\lambda_j$.
Then  by Remark \ref{rem:adjoint},
 if $\Psi_j=(1/\sqrt{\lambda_j})P_L\Phi_j$, then  $\{\Psi_j\}$
 is an orthonormal basis
 of the subspace of eigensections of $2P_LP_L^\dagger$  with nonzero eigenvalues and Alvarez boundary conditions.
 Let $\sigma=\sigma(t)$, $f=f(t)$ be one parameter families of conformal deformations; $\dot\sigma$ and $\dot f$, their derivatives.
One computes the variation of eigenvalues.
$$
\dot\lambda_j=-2\lambda_j((\dot\sigma+\dot f)\Phi_j,\Phi_j)+2\lambda_j(\dot f\Psi_j,\Psi_j)
$$
Then as in \cite[pp.\ 148-9]{A}, the corresponding variation of the determinant is given by
\begin{align*} 
\frac{d}{dt}\log\Det^\ast\dla
 &= 
  f.p.\int_\varepsilon^\infty dt\sum_{\lambda_j\neq 0}  \dot\lambda_j e^{-t\lambda_j}     \\
&= f.p.\int_\varepsilon^\infty dt\sum_{\lambda_j\neq 0} \left\{ -2\lambda_j((\dot\sigma+\dot f)\Phi_j,\Phi_j)+2\lambda_j(\dot f\Psi_j,\Psi_j)\right\} e^{-t\lambda_j} \\
&= -f.p.\int_\varepsilon^\infty dt \frac{d}{dt}\left\{  -2\tr ((\dot\sigma+\dot f)e^{-2tP_L^\dagger P_L}) + 2\tr(\dot fe^{-2t P_L P_L^\dagger})\right\}
\end{align*}
Applying Proposition \ref{prop:heat} to the heat kernel expansions for the laplacians on $D_L^{\smallA}$ and $D_{KL^\ast}^{\smallA}$,
\begin{align*}
\frac{d}{dt}\log\Det^\ast\dla &= -\frac{1}{6\pi}\int_M dA_\rho\, (6\Omega_{L,h}+R_\rho)(\dot\sigma+\dot f)+\frac{1}{6\pi}\int_M dA_\rho\, (6\Omega_{KL^\ast,(\rho h)^{-1}}+R_\rho)\dot f \\
&\qquad\qquad -\frac{1}{3\pi}\int_{\partial M}ds_\rho\, (\kappa_\rho-3\nu_{L,h}) (\dot\sigma+\dot f)+\frac{1}{3\pi}\int_{\partial M}ds_\rho\, (\kappa_\rho-3\nu_{KL^\ast,(\rho h)^{-1}})\dot f
\end{align*}
From Lemma \ref{lem:canonical} it follows that 
$
\Omega_{KL^\ast,(\rho h)^{-1}}= -(1/2)R_\rho-\Omega_{L,h}$, and
 $\nu_{KL^\ast,(\rho h)^{-1}}=\kappa_\rho-\nu_{L,h}
$.
Hence,
\begin{align}
\begin{split} \label{eqn:var}
\frac{d}{dt}\log\Det^\ast\dla &= -\frac{1}{6\pi}\int_M dA_\rho\, \left\{ 6\Omega_{L,h}(\dot\sigma+2\dot f)+R_\rho(\dot\sigma+3\dot f)\right\} \\
&\qquad\qquad -\frac{1}{3\pi}\int_{\partial M}ds_\rho\, \left\{\kappa_\rho(\dot\sigma+3\dot f)
-3\nu_{L,h} (\dot\sigma+2\dot f)\right\}
\end{split}
\end{align}
We have the following variations with respect to conformal changes.
\begin{eqnarray*}
R_{\rho}= e^{-2\sigma}(R_{\hat \rho}-2\Delta_{\hat \rho}\sigma) \qquad  
& \Omega_{L, h}&= e^{-2\sigma}(\Omega_{L,\hat h}-\Delta_{\hat \rho} f) \\
\kappa_{\rho} = e^{-\sigma}(\kappa_{\hat\rho}+\partial_{\hat n}\sigma)\qquad  &
\nu_{L, h}&=e^{-\sigma}(\nu_{L,\hat h}-\partial_{\hat n}f)
\end{eqnarray*}
Plugging these into the above, the first term on the right hand side of \eqref{eqn:var} becomes
\begin{align}
\begin{split} \label{eqn:first}
&-\frac{1}{6\pi}\int_M dA_{\hat\rho}\, \left\{ 6\Omega_{L,\hat h}
(\dot\sigma+2\dot f)+R_{\hat\rho}(\dot\sigma+3\dot f)\right\} \\
&\qquad \qquad -\frac{1}{6\pi}\int_M dA_{\hat\rho}\, \left\{ 6\nabla f\cdot \nabla \dot\sigma+12\nabla f\cdot\nabla\dot f+2\nabla\sigma\cdot\nabla\dot\sigma+6\nabla\sigma\cdot\nabla \dot f\right\} \\
&\qquad\qquad +\frac{1}{6\pi}\int_{\partial M} ds_{\hat\rho} \left\{ 12(\partial_{\hat n} f)\dot f + 2(\partial_{\hat n}\sigma)\dot\sigma +6((\partial_{\hat n} f)\dot \sigma+(\partial_{\hat n} \sigma)\dot f)\right\} 
\end{split}
\end{align}
whereas the second term on the right hand side of \eqref{eqn:var} becomes
\begin{align}
\begin{split} \label{eqn:second}
&-\frac{1}{3\pi}\int_{\partial M}ds_{\hat\rho}\, \left\{\kappa_{\hat\rho}(\dot\sigma+3\dot f)
-3\nu_{L,\hat h} (\dot\sigma+2\dot f)\right\} \\
&\qquad \qquad -\frac{1}{3\pi}\int_{\partial M}ds_{\hat\rho}\, \left\{(\partial_{\hat n}\sigma)(\dot\sigma+3\dot f)
+3(\partial_{\hat n}f)(\dot\sigma+2\dot f)\right\}
\end{split}
\end{align}
The last terms on the right hand sides of \eqref{eqn:first} and \eqref{eqn:second} cancel. The remaining terms can be integrated as in \cite{A}, giving the desired result.
\end{proof}

\begin{remark}  \label{rem:polyakov_alvarez}
 Consider the following special cases:
\begin{enumerate}
\item  $\partial M=\emptyset$.  Then the formula in \eqref{eqn:liouville} coincides with the result in \cite[Prop.\ 3.8]{F2}.  Note that there is an overall
factor of $2$, coming from the fact that the determinant $\Det^\ast \dl$, regarded as a real operator, is the square of the complex laplacian (see 
Lemma \ref{lem:real_complex}).
\item  If $L=K^q$,  $h$ the induced metric from $M$, and $f=-q\sigma$, then  \eqref{eqn:liouville} coincides with the result in \cite[eq.\ (4.29)]
{A} (see Lemma \ref{lem:canonical}).
\item If $L$ is the trivial bundle $\Ocal$ with the flat metric, then Alvarez boundary conditions amount to Dirichlet conditions on the real part and Neumann conditions on the imaginary part.  Hence, the scalar determinant is
$\Det^\ast D^{\smallA}_{\Ocal}=[{\Det}^\ast_{neu.}(\Delta)][{\Det}_{dir.}(\Delta)]$.
\item 
By Remark \ref{rem:adjoint} and Serre duality Proposition \ref{prop:serre} applied to the trivial bundle,
$$
\Det^\ast D^{\smallA}_K= \Det^\ast (2P_{\Ocal}P^\dagger_{\Ocal})=\Det^\ast (2P^\dagger_{\Ocal}P_{\Ocal})=
\Det^\ast D^{\smallA}_{\Ocal}
$$
\end{enumerate}
\end{remark}

\section{Factorization of determinants}  \label{sec:factorization}

\subsection{The generalized Dirichlet-to-Neumann operator}
In this section we assume $M$ has non-empty boundary.
Let $L\to M$ be a hermitian holomorphic bundle with framing $\tau_L$.
The following is clear.
\begin{lemma}  The real and imaginary Alvarez boundary conditions are
complimentary in the sense of \cite[Def.\ 2.12]{BFK}.
\end{lemma}

\begin{definition}  \label{def:alvarez_operator}
The Poisson operator is characterized by the condition 
$$\poisson_M(\lambda): \Bmpp\to\Omega_\R^0(M,L): (f,g)\mapsto \poisson_M(\lambda)(f,g)= \Phi
$$
where $\Phi$ satisfies $(\dl+\lambda)\Phi=0$, and $\bmpp(\Phi)=(f,g)$.
The  boundary operator is defined by
$$\A_M(\lambda):\Bmpp\to \Bmpp:   \A_M(\lambda)=J\bmp\poisson_M(\lambda)
$$
 \end{definition}
\noindent Hence, $\A_M(\lambda)$ is the analog of the Dirichlet-to-Neumann operator. Like the DN operator, $\A_M(\lambda)$ is elliptic and, by \eqref{eqn:parts} it is self-adjoint.
 In this case, however, it is a \emph{zero-th} order pseudo-differential operator instead of first order.  

In case $\lambda=0$, the  Poisson, and hence also boundary operators are not necessarily everywhere
defined nor are they a priori well-defined. This can be seen from the
integration by parts formula \eqref{eqn:parts}.  The Poisson operator is defined at $(f,g)$
only if $(f,g)$ is orthogonal to the image by $J$ of boundary values of
sections $\Phi\in \ker\dl$ satisfying  imaginary Alvarez boundary conditions.
Similarly, given any such $(f,g)$, the extension by the Poisson operator is
only well-defined up to addition of such $\Phi$.  With this in mind,  set

\begin{equation} \label{def:a_bv}
{\mathbb A}_{ M}^{\alv}=\left\{ J\bmp(\Phi) : \Phi\in \ker
\dl\ ,\ \bmpp(\Phi)=0\right\}
\end{equation}

\begin{proposition} \label{prop:extension}
On the orthogonal complement of ${\mathbb A}_{M}^{\alv}$,  the family
$\A_M(\lambda)$ extends
continuously as $\lambda\to 0$
 to an  operator $\A_M(0)=\A_M$.
\end{proposition}

\begin{proof}
Let $\{\Phi_i^{\smalla}\}_{i=1}^\infty$
 be a complete set of eigensections for $\dla$ with 
eigenvalues $\{\lambda_i\}_{i=1}^\infty$, and
 $\lambda_i=0$ if and only if $i\leq
n$.
Choose a smooth extension map $E: \Bmpp\to
L^2(M)$ satisfying $\bmpp  E=I$, $\bmp E=0$. 
 To compute $\poisson_M(\lambda)(f,g)$ we need to solve the boundary value problem
$$
(\dl+\lambda)\Phi=0 \ ,\
\bmpp(\Phi)=(f,g)
$$
on $M$.
From the definition of the extension, it suffices to solve
$$
(\dl+\lambda)\widetilde\Phi=
- (\dl+\lambda) E(f,g)\ ,\
\bmpp(\widetilde\Phi)=0
$$
for then $\Phi=
E(f,g)+\widetilde\Phi$. Moreover, by the assumption on $E$,
$J\bmp(\widetilde\Phi)=\A_M(\lambda)(f,g)$.  Now
\begin{align*}
\widetilde \Phi&=-\sum_{j=1}^\infty\frac{1}{\lambda_j+\lambda}
((\dl+\lambda)E(f,g),\Phi_j^{\smalla})_{M}
\Phi_j^{\smalla} \\
&=
-\sum_{j=1}^n
\left\{
\frac{1}{\lambda}
(\dl E(f,g),\Phi_j^{\smalla})_{M}
+
(E,\Phi_j)_{M}
\right\}
\Phi_j^{\smalla} \\
&\qquad\qquad 
 -\sum_{j=n+1}^\infty \frac{1}{\lambda_j+\lambda} 
 \left((\dl+\lambda)E(f,g),\Phi_j^{\smalla}\right)_{M}
\Phi_j^{\smalla} 
\end{align*}
By \eqref{eqn:parts}, the first sum on the right hand side reduces to
(since $\bmpp(\Phi_j^{\smalla})=0$)
\begin{align*}
&=
-\sum_{j=1}^n \left\{ \frac{1}{\lambda}
(\bm(E(f,g)),J\bm(\Phi_j^{\smalla}))
+
(E(f,g),\Phi_j^{\smalla})_{M}
\right\} \Phi_j^{\smalla} \\
&=
-\sum_{j=1}^n \left\{ \frac{1}{\lambda}
(\bmpp(E(f,g)),J\bmp(\Phi_j^{\smalla}))
+
(E(f,g),\Phi_j^{\smalla})_{M}
\right\} \Phi_j^{\smalla} \\
&=-\sum_{j=1}^n \left\{ \frac{1}{\lambda}
((f,g), J\bmp(\Phi_j^{\smalla}))
+
(E(f,g),\Phi_j^{\smalla})_{M}
\right\} \Phi_j^{\smalla} 
\end{align*}
Hence, if $(f,g)\in ({\mathbb A}_M^{\alv})^\perp$,
$$
\A_M(\lambda)(f,g)=
-\sum_{j=1}^n 
(E(f,g),\Phi_j^{\smalla})_{M}
J\bmp \Phi_j^{\smalla} 
-\sum_{j=n+1}^\infty \frac{1}{\lambda_j+\lambda} 
 \left((\dl+\lambda)E(f,g),\Phi_j^{\smalla}\right)_{M}
J\bmp\Phi_j^{\smalla} 
$$
This clearly extends continuously as $\lambda\to 0$, the second term giving the orthogonal projection to $({\mathbb A}_M^{\alv})^\perp$.
\end{proof}

\begin{example}   \label{ex:boundary_operator_disk}
Consider the disk $B_\varepsilon$ of radius $\varepsilon$ with the
euclidean metric and trivial line bundle, metric, and framing.
 Then ${\mathbb A}_{ B_\varepsilon}^{\alv}=\{0\}\oplus\R$.
By direct
computation one shows that 
\begin{equation} \label{eqn:dn_disk}
\A_{B_\varepsilon}(f,g)(\theta)=\sum_{n\neq 0}\left(\begin{matrix}0&-i\sigma(n)\\
i\sigma(n)&-\varepsilon/|n|\end{matrix}\right){\hat f(n)\choose \hat g(n)}e^{in\theta}
\end{equation}
where
\begin{equation} \label{eqn:fg}
f(\theta)=\sum_{n\in\Z}\hat f(n)e^{in\theta}\ ,\ g(\theta)=\sum_{n\neq 0}\hat 
g(n)e^{in\theta}
\end{equation}
and $\sigma(n)$ is the sign of $n$.
\end{example}

\subsection{The generalized Neumann jump operator}
Now suppose $M$ is  \emph{closed}.  Let $\Gamma\subset M$
 be a union of  simple closed disjoint curves in $M$, and define $M_\Gamma$ to be the surface with boundary obtained from $M\setminus\Gamma$ by adjoining a double cover of $\Gamma$.
We denote the connected components of $M_\Gamma$ by $\Ri$, and by $g_i$ we mean the genus of $\Ri$.
 Note that a conformal metric $\rho$ on $M$ induces one on $M_\Gamma$, and a holomorphic hermitian line bundle
 $L$  on determines one on 
  $M_\Gamma$. In both cases,  we use the same notation for the objects on $M$ and $M_\Gamma$.

 Suppose that $\tau_L$ is a framing of $L\to M_\Gamma$.  We will \emph{always} assume such framings arise from local trivializations of $L$ in a neighborhood of $\Gamma\subset M$. We have the following 
 
 \begin{lemma}
 Let $d_i$ denote the degree of $L\to \Ri$ defined by framing $\tau_L$, and let $d$ be the degree of $L\to M$.  Then $d=\sum_i d_i$.
 \end{lemma}
 
 \begin{proof}
 Let $s$ be a meromorphic section of $L$ with no zeros or poles on $\Gamma$, and let $s_i$ denote the induced meromorphic sections of $L\to \Ri$.  Clearly, $d=\deg(s)=\sum_i \deg(s_i)$.  Write $\tau_L=fs$ for a nowhere vanishing function $f$ defined in a neighborhood of $\Gamma$.
 Then the local winding number of $\tau_L$ is the sum of local winding numbers of $f$ and $s$.  On the other hand, for each component of $\Gamma$, the local winding numbers of $f$ on the two copies in $M_\Gamma$ cancel, since they are defined in terms of outward normals.  Hence,
 $$
 \sum_{i=1}\deg(\tau_L)_{\Ri}=\sum_{i=1} \deg(s_i)=d
 $$
 \end{proof}

The additivity of the Euler characteristic and Theorem \ref{thm:index} imply 
\begin{corollary} \label{cor:additive_index}
Let $M$ be a closed surface and $\Gamma\subset M$ a union
 of simple closed curves dividing $M$ into surfaces $\Ri$, $i=1,\ldots, \ell$, with boundary. Let $P_L$ be the real operator associated to $\dbar_L$ on $\Omega^0(M,L)$, and  on $\Omega^0_\R(\Ri, L)$ with  Alvarez boundary conditions defined by a framing $\tau_L$.  Then
$$
{\rm index}(P_L) =\sum_{i=1}^\ell {\rm index}(P_L)_{\Ri}
$$
\end{corollary}

 Choose an orientation for $\Gamma$.
  We define  maps
 $$\bgamma:  \Omega^0_\R(M,L)\to \Bgamma := \Omega^0_{\R}(\Gamma, \imath^\ast
L)\oplus \Omega^1_{\R}(\Gamma, \imath^\ast L)
 $$
  (and $\bgammap$, $\bgammapp$) by restriction.
 The double cover $\partial M_\Gamma\to \Gamma$ gives a \emph{diagonal} and \emph{difference} map
 \begin{align}
 \begin{split}\label{eqn:inclusion_difference_map}
 \imath_\Delta &: \Bgamma\lra\Bmgamma   \\
 \delta_\Gamma &:\Bmgamma\lra \Bgamma   
 \end{split}
 \end{align}
The  maps $\imath_\Delta$ and 
$\delta_\Gamma $ depend on the choice of orientation of $\Gamma$.  We assume that such an orientation has been fixed once and for all.

We now come the following crucial
\begin{definition} \label{def:neumann}
The Neumann jump operator
$
\N_\Gamma(\lambda): \Bgammapp\lra \Bgammapp
$
is defined by the composition:
$
\N_\Gamma(\lambda)(f,g)=\delta_\Gamma \A_{M_\Gamma}(\lambda)( \imath_\Delta(f,g))
$.
\end{definition}
\noindent
Then $\N_\Gamma(\lambda)$ is a  self-adjoint elliptic pseudo-differential operator of order zero.  
Note that $\N_\Gamma(\lambda)$ is invertible for all $\lambda>0$, since the kernel would be the boundary value of a global section in the kernel of $\dl+\lambda$.
A calculation similar to the one in \cite[Prop.\ 4.4]{BFK} leads to the following
\begin{proposition} \label{prop:symbol}
Choose coordinates with $\rho\equiv 1$ on $\Gamma$ and an appropriate gauge so that the  unitary frame associated to $\tau_L$ is parallel along $\Gamma$.  
Then the symbol of $\N_\Gamma(\lambda)$ is given by 
$$
\sigma_{\N_\Gamma(\lambda)}(x,\xi)=2(I+r_\lambda(x,\xi))a_\lambda(\xi)
$$
where $a_\lambda(\xi)$ is block diagonal with respect to the components of $\Gamma$, with blocks equal to
$$
\frac{1}{(\xi^2+\lambda)^{1/2}}\left(\begin{matrix} \lambda/2&-i\xi\\ i\xi&-2\end{matrix}\right)
$$
and $r_\lambda(x,\xi)$ is a matrix symbol with parameter $($cf.\ \cite[Def.\ 9.1]{Shu}$)$ satisfying
$$
\Vert \partial_x^m\partial_\xi^n r_\lambda(x,\xi)\Vert\leq C_{m,n}(1+|\xi|+|\lambda|^{1/2} )^{-2-n}
$$
for all $m,n\geq 0$. The same estimate holds for $\dot r_\lambda(x,\xi)=dr_\lambda(x,\xi)/d\lambda$.
\end{proposition}

Let 
\begin{equation}\label{eqn:tau}
\star: \Bgammapp\to \Bgammapp :  (f,g)\mapsto (\ast g,\ast f)
\end{equation}

\begin{corollary} \label{cor:N}
For $\lambda>0$ we have $\N_\Gamma(\lambda)=2(I+R(\lambda))A(\lambda)$, where
\begin{enumerate}
\item $A(\lambda)$ is an invertible elliptic pseudo-differential operator of order zero satisfying 
$$\star\, A(\lambda)=-A(\lambda)^{-1}\star$$
\item  $R(\lambda)$ is a pseudo-differential operator with parameter of order $-2$ with uniform bound $O(\lambda^{-1})$.
\end{enumerate}
\end{corollary}

\begin{proof}
  Define 
\begin{align}
A(\lambda)&=(\square_\Gamma+\lambda)^{-1/2}\left(\begin{matrix} \lambda/2&\ast\nabla_{\Gamma}\ast\\
-\nabla_{\Gamma} & -2\end{matrix}\right)\label{eqn:def-A} \\
R(\lambda)&=\tfrac{1}{2}\N_\Gamma(\lambda)A(\lambda)^{-1}-I \label{eqn:def-R}
\end{align}
acting on $\Bgammapp$, 
where the covariant derivatives and laplacian are with respect to the metric on $\partial M$ induced by $\rho$ and the  Chern connection.
 Then 
 (1) is clear from the definition, and (2) follows from Proposition \ref{prop:symbol} and \cite[Cor.\ 
9.1]{Shu}.
\end{proof}

As with the boundary  operator, the jump operator is not everywhere
defined for $\lambda=0$.  In order to rectify this, 
let ${\mathbb A}_\Gamma={\mathbb A}_\Gamma^{\ker}\oplus {\mathbb
A}_{\Gamma}^{\alv}$, where

\begin{align*}
{\mathbb A}_\Gamma^{\ker} &= \left\{ \bgammapp(\Phi) : \Phi\in \ker
\dl\subset \Omega^0_\R(M,L)\right\} \\
{\mathbb A}_\Gamma^{\alv} &= \left\{\delta_\Gamma J(\bmgammap(\Phi)) : 
\Phi\in \ker \dl\subset \Omega^0_\R(M_\Gamma,L)\ , \ \bmgammapp(\Phi)=0\right\} \\
\end{align*}
Notice that $
{\mathbb A}_\Gamma^{\ker}\subset  \Omega^0_\R(\Gamma, \imath^\ast L)\oplus \{0\}$, 
${\mathbb A}_\Gamma^{\alv}\subset\{0\}\oplus  \Omega^1_\R(\Gamma, \imath^\ast L)
$.
In particular, ${\mathbb A}_\Gamma^{\ker}\perp {\mathbb A}_\Gamma^{\alv}$.
Now
Propositions \ref{prop:extension} and \ref{prop:symbol} imply
\begin{proposition} \label{prop:neumann_extension}
On the orthogonal complement of ${\mathbb A}_\Gamma$,  the family
$\N_\Gamma(\lambda)$ extends
continuously as $\lambda\to 0$
 to a zero-th order  operator $\N_\Gamma(0)=\N_\Gamma$.
\end{proposition}

We also record the following
\begin{lemma} \label{lem:cokernels}
Assume $\coker P_L=\{0\}$ on $M$ and   on $M_\Gamma$.  Then
$\dim_\R {\mathbb A}_\Gamma^{\ker}=\dim_\R{\mathbb A}_\Gamma^{\alv}$.
\end{lemma}

\begin{proof}
Let
$V=\left\{ \bgammap(\Phi) : \Phi\in \ker \dl\ ,\
\bgammapp(\Phi)=0\right\}
$.
Then since any holomorphic section vanishing on $\Gamma$ must vanish
identically, we have by the assumption on cokernels  Corollary \ref{cor:additive_index},
$$\dim_\R{\mathbb A}_\Gamma^{\ker}= 2h^0(L)-\dim_\R V 
=\dim_\R\ker \dla
-\dim_\R V 
$$
On the other hand, consider the surjective map $\ker\dla\to {\mathbb A}_\Gamma^{\alv}$. Any element in the kernel corresponds to a global
holomorphic section satisfying the extra condition $\bgammapp(\Phi)=0$.  Hence,  
$$
\dim_\R\ker \dla
-\dim_\R V = \dim_\R{\mathbb A}_\Gamma^{\alv}
$$
and the result follows.
\end{proof}

\subsection{Determinants of zero-th order operators}
Let $T$ be a positive self-adjoint  elliptic pseudo-differential operator of order zero on the real Hilbert space $L^2(S^1)\oplus L^2(S^1)$ (where the $L^2$ functions are real valued).
The usual zeta regularization procedure does not apply to $T$.  In order to
 define its determinant, we need to 
choose a regularizer.  By this we mean
a positive self-adjoint elliptic pseudo-differential operator $Q$ of order $1$ on $L^2(S^1)$.  Given $Q$, we  extend it diagonally on $L^2(S^1)\oplus L^2(S^1)$ and 
denote this extended operator also by $Q$.

Next, define 
$\Log T$ as follows. 
  Let $\gamma\subset\C\setminus\{\real z\leq0\}$ be a closed curve
 containing the spectrum of  $T$.  Then define
\begin{equation} \label{eqn:log}
\Log T=\frac{1}{2\pi i}\int_{\gamma}dz (\log z) (z-T)^{-1}
\end{equation}
where $\log$ is the branch of the logarithm on $\C\setminus\{\real z\leq0\}$ with $-\pi<\arg\log z <\pi$.
Then following \cite{FG}, we set

\begin{equation} \label{eqn:def_det}
\log{\Det}_Q T= f.p.\, \tr \left(Q^{-s}\Log T\right)\bigr|_{s=0}
\end{equation}

While this definition of the determinant depends on $Q$,
 it is nevertheless very suitable for our purposes.  The main properties we will
need are summarized below. In this section and the next we will repeatedly use the fact that if bounded operators $A$ and $B$ are such that both $AB$ and $BA$ are trace class, then $\tr(AB)=\tr(BA)$ (cf.\ \cite[Cor.\ 3.8]{Simon05}).

\begin{proposition} \label{prop:det_misc}
\begin{enumerate}
\item Let $B$ be a bounded operator satisfying $BT=T^{-1}B$. Then 
$$B(\Log T)=-(\Log T)B$$
\item 
Suppose in addition that $B$ is an involution that commutes with $Q$. Then $\Det_Q T=1$.
\item  Suppose $T(\varepsilon)$ is a differentiable family of positive elliptic self-adjoint pseudo-differential operators of order zero.
If $dT(\varepsilon)/d\varepsilon$ is trace class, then 
$$
\frac{d}{d\varepsilon}\log{\Det}_Q T(\varepsilon)=\tr\Bigl(T(\varepsilon)^{-1}\frac{dT(\varepsilon)}{d\varepsilon}\Bigr)
$$
\end{enumerate}
\end{proposition}

\begin{proof}
For  (1) note that 
\begin{equation} \label{eqn:claim_T}
\frac{1}{2\pi i}\int_{\gamma}\frac{dz}{z} (\log z)(z-T)^{-1}=(\Log T) T^{-1}
\end{equation}
Indeed, from $z^{-1}(z-T)^{-1} =(z-T)^{-1}T^{-1}-z^{-1}T^{-1}$ we have
$$
\frac{1}{2\pi i}\int_{\gamma}\frac{dz}{z} (\log z)(z-T)^{-1}=\frac{1}{2\pi i}\int_{\gamma}dz (\log z)(z-T)^{-1}T^{-1}
-\frac{1}{2\pi i}\int_{\gamma}\frac{dz}{z}(\log z)T^{-1}
$$
Because of the choice of contour, the second term vanishes.
Now
$$B(z-T)=(z-T^{-1})B \ \Longrightarrow \ (z-T^{-1})^{-1}B=B(z-T)^{-1}$$
 Hence,
\begin{align*}
B(\Log T)&=\frac{1}{2\pi i}\int_{\gamma}dz (\log z)(z-T^{-1})^{-1}B 
=\frac{1}{2\pi i}\int_{\gamma}dz (\log z)(Tz-I)^{-1}TB \\
&=-\frac{1}{2\pi i}\int_{\gamma}\frac{dz}{z} (\log
z)(z^{-1}-T)^{-1}TB 
\end{align*}
Next make a change of variables $w=z^{-1}$. 
Without loss of generality, we may assume $\gamma$ is invariant under this change.
Then by \eqref{eqn:claim_T}.
$$
B(\Log T)=-\frac{1}{2\pi i}\int_{\gamma}\frac{dz}{z} (\log z)(z^{-1}-T)^{-1}TB 
=-\frac{1}{2\pi i}\int_{\gamma}\frac{dw}{w} (\log w)(w-T)^{-1}TB =
-(\Log T)B$$
For (2), it follows from (1) that
\begin{align*}
f.p.\, \tr \left(Q^{-s}\Log T\right)\bigr|_{s=0}&=
f.p.\, \tr \left(B Q^{-s}(\Log T)B\right)\bigr|_{s=0}  
=
f.p.\, \tr \left( Q^{-s}B(\Log T)B\right)\bigr|_{s=0} \\
&=-f.p.\, \tr \left(Q^{-s}\Log T\right)\bigr|_{s=0}
\end{align*}
To prove (3)
we have
\begin{align*}
\frac{d}{d\varepsilon}\Log T(\varepsilon)  &=\frac{1}{2\pi i}\int_\gamma dz\log z(z-T(\varepsilon))^{-1}\frac{dT(\varepsilon)}{d\varepsilon}(z-
T)^{-1} \\
\frac{d}{d\varepsilon}\log{\Det}_Q T(\varepsilon)&=
f.p.\bigr|_{s=0}\, \frac{1}{2\pi i}\int_\gamma dz(\log z)\tr\bigl(Q^{-s}(z-T(\varepsilon))^{-1}\frac{dT(\varepsilon)}{d\varepsilon}(z-
T(\varepsilon))^{-1}\bigr) \\
&=\frac{1}{2\pi i}\int_\gamma dz(\log z)\tr\bigl((z-T(\varepsilon))^{-2}\frac{dT(\varepsilon)}{d\varepsilon}\bigr)
\qquad\text{(since $dT(\varepsilon)/d\varepsilon$ is trace-class)} \\
&=\frac{-1}{2\pi i}\int_\gamma dz(\log z)\frac{d}{dz}\tr\bigl((z-T(\varepsilon))^{-1}\frac{dT(\varepsilon)}{d\varepsilon}\bigr) \\
&=\frac{1}{2\pi i}\int_\gamma\frac{dz}{z}\tr\bigl((z-T(\varepsilon))^{-1}\frac{dT(\varepsilon)}{d\varepsilon}\bigr) \\
&=\tr\left(T(\varepsilon)^{-1}\frac{dT(\varepsilon)}{d\varepsilon}\right)
\end{align*}
\end{proof}

\subsection{The Burghelea-Friedlander-Kappeler formula}
Continue to assume $M$ is closed with a collection of disjoint simple closed embedded curves $\Gamma$.
We apply the definition of determinant in the previous section to the Neumann jump operator. 
Let $Q$ be a positive self-adjoint elliptic pseudo-differential operator $Q$ of order $1$ on $\Omega^0_\R(\Gamma,\imath^\ast L)''$.
We  use the Hodge star to extend it as ${\rm diag}(Q,\ast Q\ast)$ 
  on  $\Bgammapp$, which we continue to denote by $Q$.
 The self-adjoint operator $\N_\Gamma(\lambda)$  has non-zero real eigenvalues for $\lambda\neq 0$, but is not positive.  Hence, we \emph{define} the logarithm and determinant by 
\begin{align*}
\Log \N_\Gamma(\lambda)&=\tfrac{1}{2}\Log ( \N_\Gamma(\lambda))^2 \\
\log \Det_Q \N_\Gamma(\lambda)&=\tfrac{1}{2}\log \Det_Q (\N_\Gamma(\lambda))^2
\end{align*}
In what follows, let $\zeta_Q(s)=\tr Q^{-s}$, 
and recall that $s=0$ is a regular value of 
 (the analytic continuation of) $\zeta_Q(s)$.

With this understood, we state the key factorization theorem (cf.\ \cite[Thm.\ A]{BFK}).
\begin{theorem}[BFK  formula] \label{thm:bfk}
For  all $\lambda>0$,
$$
\left[{\Det}(\dl+\lambda)\right]_M=c_Q\, \left[{\Det}(\dla+\lambda)\right]_{M_\Gamma}
{\Det}_Q \N_\Gamma(\lambda)
$$
where $c_Q=2^{-\zeta_Q(0)}$.
\end{theorem}

The rest of this section is devoted to the proof of this result.
First, note the following
\begin{lemma} \label{lem:trace_class}
Let $\pi_1$, $\pi_2$ be the orthogonal projections onto the first and second factors of 
$\Bgammapp$, and set 
$\dot\N_\Gamma(\lambda)=d\N_\Gamma(\lambda)/d\lambda$, and similarly for the operators $A$ and $R$.  Then
for all $\lambda>0$,
$\pi_i\N_\Gamma^{-1}\dot\N_\Gamma\pi_i$ are of order $-2$, and hence of trace class, for $i=1,2$. Moreover,
\begin{equation} \label{eqn:trace-projection}
\tr\left( \pi_1\N_\Gamma^{-1}\dot\N_\Gamma\pi_1+\pi_2\N_\Gamma^{-1}\dot\N_\Gamma\pi_2\right)
=\tr\left((I+R(\lambda))^{-1}\dot R(\lambda)\right)
\end{equation}
\end{lemma}

\begin{proof}
By Proposition \ref{prop:symbol},  $\dot R(\lambda)$ has order at most $-2$ on the circle, so  the operator on the right hand side of \eqref{eqn:trace-projection} is indeed trace class. 
In terms of the expression from Corollary \ref{cor:N},
$$
\N_\Gamma(\lambda)^{-1}\dot\N_\Gamma(\lambda)
=A(\lambda)^{-1}(I+R(\lambda))^{-1}\dot R(\lambda) A(\lambda)+ A(\lambda)^{-1}\dot A(\lambda)
$$
 It therefore suffices to prove that the operators $\pi_i A(\lambda)^{-1}\dot A(\lambda)\pi_i$, $i=1,2$, are trace class with opposite traces.  But from \eqref{eqn:def-A} we have
 $$
 A(\lambda)^{-1}\dot A(\lambda)=\tfrac{1}{2}(\square_{\Gamma}+\lambda)^{-1}\left(\begin{matrix} 1&0\\ -2\nabla_{\Gamma} &-1\end{matrix}\right)
$$
The diagonal terms have order $-2$ on the circle and so are trace class with opposite traces, and the result follows.
\end{proof}

The next result shows that in the special case of the Neumann jump operator the dependence of the determinant on the regularizer $Q$ is mild.
\begin{lemma} \label{lem:det_N}
The following hold for $\lambda$ sufficiently large:
\begin{align}
\log\Det_Q \N_\Gamma(\lambda)&=\zeta_Q(0)\log 2+\int_0^1 d\varepsilon \tr \left((I+\varepsilon 
R(\lambda))^{-1}R(\lambda)\right) \label{eqn:det-N} \\
\frac{d}{d\lambda}\log\Det_Q \N_\Gamma(\lambda)&=
\tr\left((I+R(\lambda))^{-1}\dot R(\lambda)\right) \label{eqn:miracle} 
\end{align}
\end{lemma}

\begin{proof}
From Corollary \ref{cor:N} and the definition \eqref{eqn:def_det},
\begin{equation} \label{eqn:det-R1}
\log\Det_Q \N_\Gamma(\lambda)=\zeta_Q(0)\log 2+\log\Det_Q \left((I+R(\lambda))A(\lambda)\right)
\end{equation}
 On the other hand,  $R(\lambda)A(\lambda)$ has order $-2$ and so is trace class. 
 Notice also that from Corollary \ref{cor:N} (2), $\Vert R(\lambda)\Vert=O(\lambda^{-1})$, so $I+\varepsilon R(\lambda)$ is uniformly invertible for $0\leq\varepsilon\leq 1$ and $\lambda$ sufficiently large.
  Now applying Proposition \ref{prop:det_misc} (3) to the family 
$$T(\varepsilon)=((I+\varepsilon R(\lambda))A(\lambda))^2$$
and integrating the derivative in $\varepsilon$,
 we have
\begin{equation} \label{eqn:det-R}
\log\Det_Q \left((I+R(\lambda))A(\lambda)\right)=\log\Det_Q A(\lambda)+\int_0^1 d\varepsilon \tr \left((I+\varepsilon 
R(\lambda))^{-1}R(\lambda)\right)
\end{equation}
Hence, \eqref{eqn:det-N} follows from \eqref{eqn:det-R1} and \eqref{eqn:det-R} if we 
can show $\log\Det_Q A(\lambda)=0$. 
Using  Corollary \ref{cor:N} and Proposition \ref{prop:det_misc} (1), 
$
\star\Log A(\lambda)
=-(\Log A(\lambda))\,\star\,
$, where $\star$ is given by \eqref{eqn:tau}.  Since $Q$ is a diagonal operator, $\star\, Q= Q\, \star\, $, and the claim follows from Proposition \ref{prop:det_misc} (2).
To prove \eqref{eqn:miracle}, differentiate  \eqref{eqn:det-N} to find
\begin{align*}
\frac{d}{d\lambda}\log\Det_Q \N_\Gamma(\lambda)&=
\int_0^1d\varepsilon\, \tr\left((I+\varepsilon R)^{-1}\dot R-\varepsilon(I+\varepsilon R)^{-1}\dot R(I+\varepsilon R)^{-1}R 
\right) \\
&=
\int_0^1d\varepsilon\, \tr\left((I+\varepsilon R)^{-1}\dot R-\varepsilon(I+\varepsilon R)^{-1} R(I+\varepsilon R)^{-1}\dot R 
\right) \\
&=\int_0^1d\varepsilon\, \frac{d}{d\varepsilon}\tr\left(\varepsilon (I+\varepsilon R)^{-1}\dot R\right) \\
&=
\tr\left((I+R)^{-1}\dot R\right)
\end{align*}
\end{proof}

\begin{proof}[Proof of Theorem \ref{thm:bfk}]
Let $D_L(\lambda)=D_L+\lambda$, $\dla(\lambda)=\dla+\lambda$, and ${\mathcal P}_{M_\Gamma}(\lambda)$, $\N_\Gamma(\lambda)$ the associated Poisson and Neumann jump operators.
 By the same calculation as  in \cite[Cor.\ 3.8 and Lemma 3.6]{BFK}, we have
 \begin{align}
   \N_\Gamma^{-1} \dot\N_\Gamma &={\bf b}''_\Gamma D_L^{-1} {\mathcal P}_{M_\Gamma}\imath_\Delta\label{eqn:dot-N} \\
 \frac{d}{d\lambda}\left(\log\Det D_L-\log\Det\dla\right)&=
 \tr\left( {\mathcal P}_{M_\Gamma}\imath_\Delta {\bf b}''_\Gamma D_L^{-1} \right)
 \label{eqn:dots}
 \end{align}
 where we have omitted the spectral parameter from the notation.
 For simplicity, set $P={\mathcal P}_{M_\Gamma}\imath_\Delta$ and $B={\bf b}''_\Gamma D_L^{-1}$. 
 According to \cite[Lemma 3.9]{BFK}, $PB$ is trace class.
Let $f(s)=\tr (D_L^{-s} PB)$, which is holomorphic for $\real s>0$.
 We claim that $f$ admits an analytic continuation for $\real s>-1/2$. 
 Indeed, for $\real s>0$,  
 $$
 f(s)=\tr(D_L^{-s}P B)=\tr(BD_L^{-s}P)
 =\tr(\pi_1BD_L^{-s}P\pi_1)+\tr(\pi_2BD_L^{-s}P\pi_2)
 $$
 But the operators $\pi_i B D_L^{-s} P\pi_i$ are manifestly  of order $-2-2s$ on $L^2(S^1)$; hence, the claim.
  Moreover, 
 by \eqref{eqn:trace-projection} and \eqref{eqn:miracle} we have
$
 f(0)=(d/d\lambda)\log \Det_Q\N_\Gamma(\lambda)
 $ for $\lambda$ sufficiently large.
 On the other hand, since $PB$ is trace class it is also true that $f(0)=\tr(PB)$,
 and we conclude from \eqref{eqn:dots} that
 $$
  \frac{d}{d\lambda}\log \Det_Q\N_\Gamma(\lambda)
= \frac{d}{d\lambda}\left(\log\Det D_L-\log\Det\dla\right)
 $$
 for $\lambda$ large.
 Since the determinants are analytic in $\lambda$,
 this proves the existence of the constant $c_Q$ (alternatively, notice that \eqref{eqn:miracle} holds for \emph{all} $\lambda>0$ by choosing an appropriate contour for the integral in \eqref{eqn:det-N}).  The constant $c_Q$ may  now
 be  determined by the asymptotics as $\lambda\to \infty$. By
\cite[Thm.\ 3.12 (2)]{BFK}  the claimed value for $c_Q$ holds if we show that the second term on the right hand side of \eqref{eqn:det-N} vanishes as $\lambda\to \infty$.
To see this is indeed the case, set $S(\lambda)=\square_\Gamma+\lambda$ acting on $\Bgammapp$.  It then follows from Proposition \ref{prop:symbol} that $S(\lambda)R(\lambda)$ is of order zero, and so it is bounded uniformly in $\lambda$ (cf.\ \cite[Cor.\ 9.1]{Shu}).  Now $R(\lambda)=S(\lambda)^{-1}(S(\lambda)R(\lambda))$, so $R(\lambda)$ is trace class with  $\tr|R(\lambda)|\leq C\tr S(\lambda)^{-1}$.  The eigenvalues $\{\lambda_n\}_{n=1}^\infty$ of $S(\lambda)$ have asymptotics $\lambda_n\geq an^2+\lambda-b$, for a positive constant $a$ and $n$ sufficiently large, and so by an explicit estimate $\tr S(\lambda)^{-1}$ is $O(\lambda^{-1/2})$. By the  remark in the proof of Lemma \ref{lem:det_N}, $I+\varepsilon R(\lambda)$ is uniformly invertible for $0\leq\varepsilon\leq 1$ and $\lambda$ sufficiently large.
Hence, 
$$ \left| \tr \left((I+\varepsilon 
R(\lambda))^{-1}R(\lambda)\right)\right|\leq 
C \tr\left| R(\lambda)\right| =O(\lambda^{-1/2})$$
 uniformly for $0\leq\varepsilon\leq 1$, and the result follows.
\end{proof}

\subsection{The case of zero modes} \label{sec:zero_modes}
The goal of this section is to extend the formula in Theorem \ref{thm:bfk} as $\lambda\to 0$.  We will need a preliminary 

\begin{definition} \label{def:generic_framing}
A framing $\tau_L$ near $\Gamma$ is  {\bf generic} if
$\bgammapp$ is injective on 
  $\ker \dl\subset \Omega^0_\R(M,L)$.
\end{definition}

\noindent
Note that an equivalent condition to the one above is that the difference map $\delta_\Gamma\bmgammap$ be injective on $\ker\dl^{\smalla}$ on $M_\Gamma$.  Indeed, if $\Phi$ is a global section in $\ker\dl$, then regarded as a section on $M_\Gamma$, we automatically have $\delta_\Gamma\bmgammap(\Phi)=0$.  If in addition, $\bgammapp(\Phi)=(\Phi'',0)=0$, then $\Phi\in\ker\dla$.  Conversely, if $\Phi^{\smalla}\in\ker \dl^{\smalla}$  and $\delta_\Gamma\bmgammap(\Phi^{\smalla})=0$, then
since $\delta_\Gamma\bmgammapp(\Phi^{\smalla})=0$ automatically,  it extends to a global section on $M$.

\begin{theorem} \label{thm:bfk_generic}
For a given framing $\tau_L$ near $\Gamma$, 
let $\{\Phi_i\}$ $($resp.\ $\{ \phia_i\}$$)$
 be a basis for $\ker\dl$ on $M$ $($resp.\  for $\ker\dla$ on $M_\Gamma$$)$.  Assume the framing is generic
 in the sense of Definition \ref{def:generic_framing}. Then
$$
\left[\frac{{\Det}^\ast\dl}{\det(\Phi_i, \Phi_j)}\right]_M=c_Q\, \left[\frac{{\Det}^\ast\dla}{\det(\phia_i,\phia_j)}\right]_{M_\Gamma}
\frac{\det(\delta_\Gamma\phia_i,\delta_\Gamma\phia_j)_\Gamma }{\det(\Phi_i'',\Phi_j'')_\Gamma}\, {\Det}_Q^\ast \N_\Gamma
$$
where  $\N_\Gamma=\N_\Gamma(0)$ is the operator defined on the orthogonal complement of ${\mathbb A}_\Gamma$ in Proposition \ref{prop:neumann_extension}.
\end{theorem}

\begin{proof}
We apply Theorem \ref{thm:bfk} as $\lambda\downarrow
 0$. 
 By the definition of zeta regularization,
\begin{align*}
\log\Det(\dl+\lambda)&=(\log\lambda)\dim_\R\ker\dl + \log\Det^\ast\dl
+ o(1) \\
\log\Det(\dla+\lambda)&=(\log\lambda)\dim_\R\ker\dla + \log\Det^\ast\dla
+ o(1)
\end{align*}
on $M$ and $M_\Gamma$ with Alvarez boundary conditions.  Let $m= \dim_\R\ker\dl$ on $M$, and $n=\dim_\R\ker\dla$ on $M_\Gamma$.
Hence, it suffices to compute 
 $\lim_{\lambda\to 0}\left\{\log\Det_Q \N_\Gamma(\lambda)+(n-m)\log\lambda\right\}$.
The key point is that there are small eigenvalues of 
$\N_\Gamma(\lambda)$, $\mu_j(\lambda)\to 0$, $j=1,\ldots, m$, 
corresponding to global holomorphic sections of $L$, and large 
eigenvalues $\nu_j(\lambda)\to +\infty$, $j=1,\ldots, n$, 
corresponding to $\ker \dla$. 
Moreover, it follows easily from the definition that
\begin{equation} \label{eqn:small_lambda}
\log\Det  \N_\Gamma(\lambda)=
\log(\mu_1(\lambda)\cdots\mu_m(\lambda))+
\log(\nu_1(\lambda)\cdots\nu_n(\lambda))+
\log{\Det}_Q^\ast \N_\Gamma+o(1)
\end{equation}
 We need therefore to compute the
 contribution from both the $\{\mu_i\}$ and $\{\nu_i\}$.

Let $\mu_1(\lambda),\ldots, \mu_m(\lambda)$ be the small eigenvalues of
 $\N_\Gamma(\lambda)$, and let $\{\beta_j(\lambda)\}_{j=1}^m$ be
orthonormal with eigenvalues $\mu_j(\lambda)$.  
Let $\{\Phi_j\}_{i=1}^\infty$
 be a complete set of eigensections for $\dl$ on $M$
with eigenvalues $\{\lambda_j\}_{j=1}^\infty$,
 $\lambda_j=0$ if and only if $j\leq
m$.
Let  $\pi: \Bgammapp\to \Bgammapp$ orthogonal projection to ${\mathbb A}_\Gamma^{\ker}$.
Then  we compute
$$
{\mathcal N}^{-1}_\Gamma(\lambda)=
\left(
\begin{matrix}
\tfrac{1}{\lambda}A_1+\pi B_1(\lambda)\pi & \pi B_1(\lambda)\pi^\perp  \\
\pi^\perp B_1(\lambda)\pi  & \pi^\perp B_1(\lambda)\pi^\perp 
\end{matrix}
\right)
$$
where $A_1,B_1(\lambda) :  L^2(\Gamma)\to L^2(\Gamma)$ are given by 
\begin{align}
A_1(F,G) &= \sum_{j=1}^m \left((F,G),\bgammapp(\Phi_j)\right)_\Gamma
\bgammapp(\Phi_j) \label{eqn:A}\\
B_1(\lambda)(F,G) &= \sum_{j=m+1}^\infty \frac{1}{\lambda_j+\lambda} 
 \left((F,G),\bgammapp(\Phi_j)\right)
 \bgammapp(\Phi_j)
 \notag
\end{align}
To see this, let $\Phi$ be a section of $L\to M_\Gamma$, $(\dl+\lambda)\Phi=0$, with $(F,G)=\delta_\Gamma J\bmgammap(\Phi)$, and $\delta_\Gamma \bmgammapp(\Phi)=0$.  Then by \eqref{eqn:parts},
\begin{align*}
\Phi = \sum_{j=1}^\infty (\Phi, \Phi_j)_{M_\Gamma}\Phi_j
&= \sum_{j=1}^\infty\frac{1}{\lambda_j+\lambda} (\Phi, (\dl+\lambda)\Phi_j)_{M_\Gamma}\Phi_j \\
&=- \sum_{j=1}^\infty\frac{1}{\lambda_j+\lambda} (\bmgamma\Phi, J\bmgamma\Phi_j)\Phi_j \\
&=- \sum_{j=1}^\infty\frac{1}{\lambda_j+\lambda} (\delta_\Gamma\bgamma\Phi, J\bgamma\Phi_j)\Phi_j \\
&=- \sum_{j=1}^\infty\frac{1}{\lambda_j+\lambda} (\delta_\Gamma\bgammap\Phi, J\bgammapp\Phi_j)\Phi_j \\
&=\sum_{j=1}^\infty\frac{1}{\lambda_j+\lambda} (\delta_\Gamma J\bgammap\Phi, \bgammapp\Phi_j)\Phi_j \\
&=\sum_{j=1}^\infty\frac{1}{\lambda_j+\lambda} ((F,G), \bgammapp\Phi_j)\Phi_j
\end{align*}
and computing $\bgammapp(\Phi)$ gives the result.
We wish to relate the eigenvalues of $A_1$ to the $\mu_j(\lambda)$.  Since
$$
{\mathcal N}^{-1}_\Gamma(\lambda)\beta_j(\lambda)=\mu_j^{-1}(\lambda)
\beta_j(\lambda)
$$
we have
\begin{align}
\frac{1}{\lambda}A_1\beta_j(\lambda)+\pi B_1(\lambda)\beta_j(\lambda) &=
\mu_j^{-1}(\lambda) \pi\beta_j(\lambda) \label{E:a_ker}\\
\pi^\perp B_1(\lambda)\beta_j(\lambda) &= \mu_j^{-1}(\lambda)
\pi^\perp\beta_j(\lambda) \notag
\end{align}
Since $B_1(\lambda)$ is uniformly bounded as $\lambda\downarrow 0$, it follows that
$
\Vert \pi^\perp\beta_j(\lambda)\Vert_{L^2(\Gamma)}\leq C
\mu_j(\lambda)
$,
for $C$ independent of $\lambda$.  In particular, $\Vert
\pi\beta_j(\lambda)\Vert_{L^2(\Gamma)}\to 1$ as $\lambda\downarrow 0$, and so
 (after passing to a sequence $\lambda_k\downarrow 0$)
 there exist limits $\{\beta_j(0)\}$ which give a basis for ${\mathbb A}_\Gamma^{\ker}$.  If we 
let $v_j$ be an orthonormal basis for ${\mathbb A}_\Gamma^{\ker}$ such that $A_1
v_j=\sigma_j v_j$, and write
$$
\pi\beta_j(\lambda)= \sum_{k=1}^m C_{jk}(\lambda)v_k 
$$
then the (subsequential) limit $C_{jk}(0)$ exists and is nonsingular.
From \eqref{E:a_ker} we have
\begin{equation} \label{E:aa_ker}
\left\Vert
A_1\pi\beta_j(\lambda)-\frac{\lambda}{\mu_j(\lambda)}
\pi\beta_j(\lambda)\right\Vert_{L^2(\Gamma)}\leq
C\lambda\
\end{equation} 
In terms of the basis $\{v_j\}$, 
$$
A_1\pi\beta_j(\lambda)-\frac{\lambda}{\mu_j(\lambda)}\pi\beta_j(\lambda)
=\sum_{k=1}^m C_{jk}(\lambda)\left(
\sigma_k-\frac{\lambda}{\mu_j(\lambda)}\right)v_k
$$
so by \eqref{E:aa_ker},
$
C_{jk}(\lambda)(\sigma_k-(\lambda/\mu_j(\lambda)))\to 0
$,
for all $j,k$.
  Since $(C_{jk})$ is non-singular, for each $j$, $C_{jk}(0)\neq 0$ for some $k$.  Hence,
$
\sigma_k^{-1} =\lim_{\lambda\downarrow
0}\mu_j(\lambda)/\lambda=\hat\mu_j
$
exists for each $j$, with $C_{jk}\sigma_k=\hat\mu_j^{-1}C_{jk}$.  Again using the fact that $(C_{jk})$ is non-singular,
we have
$$\log\det A_1+\log(\prod \mu_j(\lambda))=m\log\lambda
+o(1)
$$
Finally, note that by choosing 
$\bgammapp(\Phi_j)$, $j=1, \ldots, m$,  as a basis in \eqref{eqn:A},
 we have 
 $$\det A_1=\det(
\bgammapp(\Phi_i),\bgammapp(\Phi_j))$$ 
\begin{equation} \label{eqn:A1}
\log(\prod \mu_j(\lambda))=m\log\lambda-
 \log\det(
\bgammapp(\Phi_i),\bgammapp(\Phi_j))+o(1)
\end{equation}

Let $\nu_1(\lambda),\ldots, \nu_n(\lambda)$ be the divergent eigenvalues of
 $\N_\Gamma(\lambda)$, and let $\{\beta_j^{\smalla}(\lambda)\}_{j=1}^n$ be
orthonormal with eigenvalues $\nu_j(\lambda)$.  
Let $\{\Phi_i^{\smalla}\}_{i=1}^\infty$
 be a complete set of eigensections for $\dla$ on $M_\Gamma$ with
eigenvalues $\{\lambda_i\}_{i=1}^\infty$, and
 $\lambda_i=0$ if and only if $i\leq
n$.
Let  $\pi: \Bgammapp\to \Bgammapp$ be orthogonal projection to ${\mathbb A}^{\alv}_\Gamma$.
We also choose a smooth extension map $E: \Bgammapp\to
L^2(M_\Gamma)$ satisfying $\bgammapp  E=I$, $\bgammap E=0$. 
Then as above we compute
$$
\N_\Gamma(\lambda)=
\left(
\begin{matrix}
\tfrac{1}{\lambda}A_2(\lambda)+\pi B_2(\lambda)\pi & \pi B_2(\lambda)\pi^\perp  \\
\pi^\perp B_2(\lambda)\pi  & \pi^\perp B_2(\lambda)\pi^\perp 
\end{matrix}
\right)
$$
where $A_2(\lambda),B_2(\lambda) : 
 L^2(\Gamma)\to L^2(\Gamma)$ are given by 
\begin{align}
A_2(\lambda)(f,g) &= -\sum_{j=1}^n\left\{ 
\left((f,g),\delta_\Gamma J\bmgammap(\Phi_j^{\smalla})\right)
+\lambda\left(E(f,g),\Phi_j^{\smalla}\right)_M\right\} \delta_\Gamma J\bmgammap(\Phi_i^{\smalla})
 \label{eqn:A2}\\
B_2(\lambda)(f,g) &= -\sum_{j=n+1}^\infty \frac{1}{\lambda_j+\lambda} 
 \left((\dl+\lambda)E(f,g),\Phi_j^{\smalla}\right)_M
\delta_\Gamma J\bmgammap(\Phi_j^{\smalla})
 \notag
\end{align}
To see this, note that to compute $\N_\Gamma(\lambda)(f,g)$ we need to solve the boundary value problem
$$
(\dl+\lambda)\Phi=0 \ ,\
\bmgammapp(\Phi)=\imath_\Delta(f,g)
$$
on $M_\Gamma$.
From the definition of the extension, it suffices to solve
$$
(\dl+\lambda)\widetilde\Phi=
- (\dl+\lambda) E(f,g)\ ,\
\bmgammapp(\widetilde\Phi)=0
$$
for then $\Phi=
E(f,g)+\widetilde\Phi$, and by the assumption on $E$ the jump in
$\bmgammap(\widetilde\Phi)$ gives $\N_\Gamma(\lambda)(f,g)$.  Now
\begin{align*}
\widetilde \Phi&=-\sum_{j=1}^\infty\frac{1}{\lambda_j+\lambda}
((\dl+\lambda)E(f,g),\Phi_j^{\smalla})_{M_\Gamma}
\Phi_j^{\smalla} \\
&=
-\sum_{j=1}^n
\left\{
\frac{1}{\lambda}
(\dl E(f,g),\Phi_j^{\smalla})_{M_\Gamma}
+
(E,\Phi_j)_{M_\Gamma}
\right\}
\Phi_j^{\smalla} \\
&\qquad\qquad 
 -\sum_{j=n+1}^\infty \frac{1}{\lambda_j+\lambda} 
 \left((\dl+\lambda)E(f,g),\Phi_j^{\smalla}\right)_{M_\Gamma}
\Phi_j^{\smalla} 
\end{align*}
By \eqref{eqn:parts}, the first term on the right hand side is
(since $\bmgammapp(\Phi_j^{\smalla})=0$)
\begin{align*}
&=
-\sum_{j=1}^n \left\{ \frac{1}{\lambda}
(\bmgamma(E(f,g)),J\bmgamma(\Phi_j^{\smalla}))
+
(E(f,g),\Phi_j)_{M_\Gamma}
\right\} \Phi_j^{\smalla} \\
&=
-\sum_{j=1}^n \left\{ \frac{1}{\lambda}
(\bmgammapp(E(f,g)),J\bmgammap(\Phi_j^{\smalla}))
+
(E(f,g),\Phi_j^{\smalla})_{M_\Gamma}
\right\} \Phi_j^{\smalla} \\
&=
-\sum_{j=1}^n \left\{ \frac{1}{\lambda}
(\bgammapp(E(f,g)),\delta_\Gamma J\bmgammap(\Phi_j^{\smalla}))
+
(E(f,g),\Phi_j^{\smalla})_{M_\Gamma}
\right\} \Phi_j^{\smalla} \\
&=-\sum_{j=1}^n \left\{ \frac{1}{\lambda}
((f,g),\delta_\Gamma J\bmgammap(\Phi_j^{\smalla}))
+
(E(f,g),\Phi_j^{\smalla})_{M_\Gamma}
\right\} \Phi_j^{\smalla} 
\end{align*}
We again relate the eigenvalues of $A_2(0)$
 to the $\nu_j(\lambda)$.  Since
$
\N_\Gamma(\lambda)\beta_j^{\smalla}(\lambda)=\nu_j(\lambda)
\beta_j^{\smalla}(\lambda)
$,
we have
\begin{align}
\frac{1}{\lambda}A_2(\lambda)\beta_j^{\smalla}(\lambda)+\pi 
B_2(\lambda)\beta_j^{\smalla}(\lambda) &=
\nu_j(\lambda) \pi\beta_j^{\smalla}(\lambda) \label{E:a}\\
\pi^\perp B_2(\lambda)\beta_j^{\smalla}(\lambda) &= \nu_j(\lambda)
\pi^\perp\beta_j^{\smalla}(\lambda) \notag
\end{align}
Since $B_2(\lambda)$ 
is uniformly bounded as $\lambda\downarrow 0$, it follows that
$
\Vert \pi^\perp\beta_j^{\smalla}(\lambda)\Vert_{L^2(\Gamma)}\leq C
\nu_j^{-1}(\lambda)
$,
for $C$ independent of $\lambda$.  In particular, $\Vert
\pi\beta_j^{\smalla}(\lambda)
\Vert_{L^2(\Gamma)}\to 1$ as $\lambda\downarrow 0$,
 and so the (sequential) limits $\{\beta_j^{\smalla}(0)\}$
 give a basis for ${\mathbb A}^{\alv}_\Gamma$.  If we 
let $v_j$ be an orthonormal basis
 for ${\mathbb A}^{\alv}_\Gamma$ such that $A_2(0)v_j=\sigma_j v_j$, and write
$$
\pi\beta_j^{\smalla}(\lambda)= \sum_{k=1}^n C_{jk}(\lambda)v_k \ ,
$$
then $C_{jk}(0)$ exists and is nonsingular.
From \eqref{E:a} we have
\begin{equation} \label{E:aa}
\left\Vert
A_2(0)\pi\beta_j^{\smalla}(\lambda)-\lambda\nu_j(\lambda)\pi\beta_j^{\smalla}
(\lambda)
\right\Vert_{L^2(\Gamma)}\leq
C\lambda\
\end{equation} 
In terms of the basis $\{v_j\}$, 
$$
A_2(0)\pi\beta_j^{\smalla}(\lambda)-\lambda\nu_j(\lambda)\pi
\beta_j^{\smalla}(\lambda)
=\sum_{k=1}^n C_{jk}(\lambda)\left(
\sigma_k-\lambda\nu_j(\lambda)\right)v_k
$$
so by \eqref{E:aa},
$
C_{jk}(\lambda)( \sigma_k-\lambda\nu_j(\lambda))\to 0
$
for all $j,k$.  As before,
$
\lim_{\lambda\downarrow
0}\lambda\nu_j(\lambda)=\hat\nu_j
$
exists for each $j$, and $C_{jk}\sigma_k=\hat\nu_j C_{jk}$ for all $j,k$.
Hence, $\log\det A_2(0)=\log(\prod \nu_j(\lambda))+m\log\lambda
+o(1)$.
Finally, note that by choosing 
$\delta_\Gamma J\bmgammap(\Phi_j^{\smalla})$ as a basis in \eqref{eqn:A2},
 we have 
 $$\det A_2(0)=\det(\delta_\Gamma\bmgammap(\Phi_i^{\smalla}),\delta_\Gamma\bmgammap(\Phi_j^{\smalla}))$$
\begin{equation} \label{eqn:A22}
\log(\prod \nu_j(\lambda))=-m\log\lambda+
 \log\det(\delta_\Gamma\bmgammap(\Phi_i^{\smalla}),\delta_\Gamma\bmgammap(\Phi_j^{\smalla}))
+o(1)
\end{equation}
Putting together \eqref{eqn:small_lambda}, \eqref{eqn:A1} and \eqref{eqn:A22} gives the result.
\end{proof}

We will later use the following special case of 
Theorem \ref{thm:bfk_generic}:  let $\Gamma$ be a 
simple closed connected curve separating $M$
 into components $\Rone$ and $\Rtwo$. 
 Then for any choice of bases $\{\Phi_i\}_{i=1}^{m}$  for $\ker\dl$ on $M$, and
 $\{\phiaone_i\}_{i=1}^{m_1}$, and $\{\phiatwo_i\}_{i=1}^{m_2}$ for $\ker\dla$ on $\Rone$ and $\Rtwo$, we have

\begin{equation} \label{eqn:sep_curve}
\left[\frac{{\Det}^\ast\dl}{\det(\Phi_i, \Phi_j)}\right]_M=c_Q\, \left[\frac{{\Det}^\ast\dla}{\det(\phiaone_i,\phiaone_j)}\right]_{\Rone} 
\left[ \frac{{\Det}^\ast\dla}{\det(\phiatwo_i,\phiatwo_j)}\right]_{\Rtwo} \\
\frac{\det(\phia_i,\phia_j)_\Gamma}{\det(\Phi_i'',\Phi_j'')_\Gamma}\, {\Det}_Q^\ast \N_\Gamma
\end{equation}
where
$$
\Phi_i^{\smalla}=\begin{cases} \Phi_i^{\smalla,\smallone} & 1\leq i\leq m_1 \\ \Phi_{i-m_1}^{\smalla,\smalltwo} & m_1<  i\leq m_1+m_2 \end{cases}
$$
extended by zero to the whole surface $M_\Gamma$.

Actually, for the purpose of degeneration it will be useful to also have a 
slightly modified version of \eqref{eqn:sep_curve} in the case where the trivialization 
$\tau_L$ is in fact the restriction of a  global holomorphic section. 
 This is not a generic situation in the sense of Definition \ref{def:generic_framing}, 
since the global section $\tau_L$ also satisfies
 Alvarez boundary conditions, and hence 
$\det(\Phi_i'',\Phi_j'')_\Gamma=0$ for any basis.  Similarly, since the jump of $\tau_L$ is trivial,  
$\det(\delta_\Gamma\phia_i,\delta_\Gamma\phia_j)$ also vanishes.  This motivates the following

\begin{definition} \label{def:good_framing}
Let $\tau_L$ be a global holomorphic section of $L\to M$, nowhere vanishing near $\Gamma$.  We call the framing $\tau_L$  {\bf good} if
the kernel of $\bgammapp$  on 
  $\ker \dl\subset \Omega^0_\R(M,L)$ is precisely the $\R$-span of $\tau_L$.
We say that  bases  $\{\Phi_i\}_{i=1}^{m}$, $\{\phiaone_i\}_{i=1}^{m_1}$, and $\{\phiatwo_i\}_{i=1}^{m_2}$ for $\ker\dl$ on $M$ and
for $\ker\dla$ on $\Rone$ and $\Rtwo$,
 are {\bf adapted} to $\tau_L$ if
$\Phi_1=\tau_L$,
$\phiaone_1=\tau_L\bigr|_{\Rone}$, 
$\phiatwo_1=\tau_L\bigr|_{\Rtwo}
$.
\end{definition}

\noindent
For adapted bases the notation ${\det}^\ast(\Phi_i,\Phi_j)_\Gamma$ will by definition denote the determinant  of the  $(11)$-minor of $
(\Phi_i,\Phi_j)_\Gamma$.  Similarly for $\phia_i$.
 Then after some linear algebra we have

 \begin{theorem} \label{thm:bfk_adapted}
Let $\tau_L$ be a global holomorphic section giving a  framing of $L$ near a simple closed separating curve $\Gamma$, and
let $\{\Phi_i\}$ $($resp.\ $\{\phiaone_i, \phiatwo_j\}$$)$ be an adapted basis for $\ker\dl$ on $M$ $($resp.\ on $\Ronetwo$ with Alvarez 
boundary conditions$)$.  Assume the framing is  good in the sense of Definition \ref{def:good_framing}. Then
$$
\left[\frac{{\Det}^\ast\dl}{\det(\Phi_i, \Phi_j)}\right]_M=c_Q\, 
\left[\frac{{\Det}^\ast\dla}{\det(\phiaone_i,\phiaone_j)}\right]_{\Rone}
\left[ \frac{{\Det}^\ast\dla}{\det(\phiatwo_i,\phiatwo_j)}\right]_{\Rtwo}  
\frac{{\det}^\ast(\phia_i,\phia_j)_\Gamma }{{\det}^\ast(\Phi_i'',\Phi_j'')_\Gamma}\, {\Det}_Q^\ast \N_\Gamma
$$
\end{theorem}
 \begin{example}
 As a special case of Theorem \ref{thm:bfk_generic},  consider the $2$-sphere $S^2_R$ of radius $R$ cut along an equator $\Gamma$ into two copies of the hemisphere $H^2_R$.  Choose the canonical bundle $K$ with the canonical framing.  Then $\ker \dl$ and $\ker\dla$ are both trivial, so the condition in Definition \ref{def:generic_framing} is trivially satisfied.  Moreover, it is easy to see that $c_Q\Det\N_\Gamma=1$.  Using this and Remark \ref{rem:polyakov_alvarez} (4), 
  $$
 \left[\Det^\ast D_{\Ocal} \right]_{S^2_R} =\left[\Det D_K\right]_{S^2_R}= c_Q  \left[\Det D_K^{\smallA}\right]_{H^2_R}^2 \Det\N_\Gamma =\left[\Det^\ast D_{\Ocal}^{\smallA}\right]_{H^2_R}^2 
 $$
 and so by Lemma \ref{lem:real_complex}  and Remark \ref{rem:polyakov_alvarez} (3), 
we obtain the well-known formula 
$$ \left[\Det^\ast\Delta\right]_{S^2_R}=\left[\Det^\ast_{neu.}\Delta\right]_{H^2_R}\left[\Det_{dir.} \Delta\right]_{H^2_R}
$$
 \end{example}
 
\section{Asymptotics of Determinants}  \label{S:degeneration}

\subsection{Asymptotics of the generalized Neumann jump operator}
\label{sec:asy_neumann}
The goal of this section is to prove the following.
Let $M$ be  a closed Riemann surface of genus $g$, and choose a  coordinate
 neighborhood $B$  with coordinate $z$ centered at $p\in M$. 
Let $B_\varepsilon=\{|z|<\varepsilon\}$, and set 
$R_\varepsilon=M\setminus B_\varepsilon$. 
 Let $L\to M$ 
be a hermitian holomorphic line bundle of degree $d$ with a global holomorphic section $\tau_L$ that is nowhere vanishing on $B$.  Also, 
assume $\coker P_L=\{0\}$ on $M$ and  on $R_\varepsilon$, and that $\rho\equiv 1$ and  $\Vert\tau_L\Vert=1$ on $B$.

\begin{proposition} \label{prop:dn_pinch_one}
 If $\N_{\Gamma_\varepsilon}$ denotes the Neumann jump operator with respect to Alvarez boundary conditions defined by a \emph{global} section $\tau_L
$.  Then
 as $\varepsilon\to 0$,
$$
\log{\Det}^\ast_Q \N_{\Gamma_\varepsilon}\longrightarrow (\zeta_Q(0)-4h^0(L)+2)\log 2
$$
\end{proposition}

By direct computation, as in \cite{W} one proves 
\begin{lemma} \label{lem:dn_degeneration}
For $1/2\geq \varepsilon>0$, 
$
\A_{R_\varepsilon}=S_\varepsilon+ \varepsilon U_\varepsilon
\A_{R_1}(I+T_\varepsilon\A_{R_1})^{-1}U_\varepsilon
$,
where 
\begin{align*}
S_\varepsilon(f,g)(\theta) &=
 \sum_{n\neq
0}\left(\frac{\varepsilon^n-\varepsilon^{-n}}{\varepsilon^n+\varepsilon^{-n}}\right)
\left(\begin{matrix}0& -i\\ i&-\varepsilon/n\end{matrix}\right){\hat
f(n)\choose \hat g(n)}e^{in\theta} \\
U_\varepsilon (f,g)(\theta)&=
 \sum_{n\neq
0}\frac{2}{\varepsilon(\varepsilon^n+\varepsilon^{-n})}
{\hat f(n)\choose\varepsilon \hat g(n)}e^{in\theta} \\
T_\varepsilon (f,g)(\theta)&=
 \sum_{n\neq
0}\left(\frac{\varepsilon^n-\varepsilon^{-n}}{\varepsilon^n+\varepsilon^{-n}}\right)
\left(\begin{matrix}1/n& -i\\ i&0\end{matrix}\right){\hat
f(n)\choose  \hat g(n)}e^{in\theta}
\end{align*}
for functions $f,g$ in \eqref{eqn:fg}.
\end{lemma}

We also note the following estimates.

\begin{lemma}  Assume $1/2\geq \varepsilon>0$.
\begin{enumerate}
\item  $(A_\varepsilon-S_{\varepsilon})$ is trace class with norm bounded
by $8\varepsilon^2$.
\item $U_\varepsilon$ is trace class with uniformly bounded norm.
\item If $T_0$ is defined by
$$
T_0 (f,g)(\theta)=
 \sum_{n\neq
0}
\left(\begin{matrix}-1/|n|& i\sigma(n)\\-i\sigma(n)&0\end{matrix}\right){\hat
f(n)\choose \hat g(n)}e^{in\theta}
$$
then $(T_\varepsilon-T_0)$ is trace class with norm bounded
by $8\varepsilon^2$.
\end{enumerate}
\end{lemma}

\begin{lemma} \label{lem:invertible}
For $\varepsilon>0$ sufficiently small , $I+T_\varepsilon \A_{R_1}$  is uniformly invertible on the orthogonal complement of 
${\mathbb A}_{\Gamma}$.
\end{lemma}

\begin{proof}
It suffices to show that $I+T_0 \A_{R_1}$ has no kernel on ${\mathbb A}_{\Gamma}^\perp$.  But by a direct computation, if 
$(f,g)=-T_0\A_{R_1}(f,g)$, then $\poisson_{R_1}(f,g)$ extends to a global section in $\ker\dl$. 
\end{proof}

\begin{proof}[Proof of Proposition \ref{prop:dn_pinch_one}]
By Lemma \ref{lem:dn_degeneration} we have on the orthogonal complement of ${\mathbb A}_{\Gamma_\varepsilon}$:
\begin{align*}
\Log \N_{\Gamma_\varepsilon}&= \log 2 + \tfrac{1}{2}\Log\left(\tfrac{1}{2}\N_{\Gamma_\varepsilon}\right)^2 
=  \log 2 +\tfrac{1}{2}\Log\left(I+C(\varepsilon)\right) \\
\log{\Det}^\ast_Q \N_{\Gamma_\varepsilon}&=(\zeta_Q(0)-\dim_\R {\mathbb A}_{\Gamma_\varepsilon})\log 2
+\tfrac{1}{2}\log\Det_Q^\ast(I+C(\varepsilon))
\end{align*}
More precisely, assume the orientation of $\Gamma$ is chosen 
to agree with $\partial R_\varepsilon$, and let $f,g$ be functions as in
\eqref{eqn:fg}. Let $\Sigma$ be the involution that sends $\hat f(n)\mapsto \hat f(-n)$ and $\hat g(n)\mapsto \hat g(-n)$.
 Now using \eqref{eqn:dn_disk},
\begin{align*}
\N_{\Gamma_\varepsilon}&=\A_{R_\varepsilon}+\Sigma\circ\A_{B_\varepsilon}\circ
\Sigma \\
\N_{\Gamma_\varepsilon}{f\choose g}
&=\sum_{n\neq 0}\left[\left(\begin{matrix}0& i\sigma(n)\\ 
-i\sigma(n)&\varepsilon/|n|\end{matrix}\right)+\Sigma\circ
\left(\begin{matrix}0& -i\sigma(n)\\ i\sigma(n)&-\varepsilon/|n|\end{matrix}\right)
\circ\Sigma\right]{\hat f(n)\choose \hat
g(n)} e^{in\theta} +\{\text{trace class}\} \\
&=\sum_{n\neq 0} 2{i\sigma(n)\hat g(n) \choose -i\sigma(n)\hat f(n)}e^{in\theta}
+\{\text{trace class}\} \\
\N_{\Gamma_\varepsilon}^2&=4I+\{\text{trace class}\}
\end{align*}
Now by Lemma \ref{lem:cokernels},
$$
\dim_\R {\mathbb A}_{\Gamma_\varepsilon}=\dim_\R\ker\dl -1+ \dim_\R\ker\dla-1
=2\dim_\R\ker\dl-2
=4h^0(L)-2
$$
Since $C(\varepsilon)\to 0$ in trace, the result follows from Proposition  \ref{prop:det_misc} (3).
\end{proof}

Next we assume $L$ has a framing given by a global meromorphic section with simple pole at $p$.  It is easy to see that for an appropriate annular coordinate  on the disk $L$ is isomorphic as a framed bundle to the canonical bundle with canonical framing.  In this case we  have the following asymptotics.
\begin{proposition} \label{prop:dn_pinch_two}
 If $\N_{\Gamma_\varepsilon}$ denotes the Neumann jump operator with respect to Alvarez boundary conditions defined by a \emph{global} meromorphic section $\tau_L
$ with simple pole at $p$, then
 as $\varepsilon\to 0$,
$$
\log{\Det}^\ast_Q \N_{\Gamma_\varepsilon}+\log(\varepsilon/2) \longrightarrow (\zeta_Q(0)-4h^0(L)-2)\log 2
$$
\end{proposition}

\begin{proof}
The computation is nearly identical to the one above, except now 
$\dim_\R {\mathbb A}_{\Gamma_\varepsilon}=4h^0(L)+1$, 
and the constant mode $(1,0)\not\in  {\mathbb A}^{\rm ker}_{\Gamma_\varepsilon}$.  By assumption, $\A_{R_\varepsilon}(1,0)=(0,0)$, and by direct computation for the canonical bundle on the disk, $\A_{B_\varepsilon}(1,0)=(2/\varepsilon,0)$.  Factoring this out from the determinant, the result follows.
\end{proof}

\subsection{Admissible metrics and asymptotics of $S(\sigma,f)$} \label{sec:arakelov}

Recall the definition of the Arakelov metric (cf.\ \cite{Ar, Fg,  CMP, F2}).  Given  a compact Riemann surface $M$ of genus $g\geq 1$, let 
$\{A_i,B_i\}_{i=1}^g$ be a symplectic set of generators of $H_1(M)$ and
choose $\{\omega_i\}_{i=1}^g$ to be a basis of abelian differentials normalized such that
$\int_{A_i}\omega_j=\delta_{ij}$.  Let 
$
\Omega_{ij}=\int_{B_i}\omega_j
$
 be the associated period matrix with theta function $\vartheta$.
  Set 
\begin{equation*} \label{E:bergman}
\mu=\frac{i}{2g}\sum_{i,j=1}^g (\imag \Omega)^{-1}_{ij}\omega_i\wedge\overline\omega_j
\end{equation*}
Then $\int_M\mu=1$.  The {\bf Arakelov-Green's function}
 $G(z,w)$ is symmetric with a zero of order one along the diagonal satisfying 
$
\partial\bar\partial\log G(z,w)=(\pi i) \mu
$, for $z\neq w$, normalized by
\begin{equation} \label{eqn:normalization}
\int_M\mu(z) \log G(z,w)= 0 
\end{equation}
  The {\bf Arakelov metric} $\rho_{\smallAr}=\rho_{\smallAr}(z)|dz|^2$  is defined by
\begin{equation}
\log \rho_{\smallAr}(z)=2\lim_{w\to z}\left\{ \log G(z,w)-\log |z-w|\right\} \label{E:metric}
\end{equation}

A hermitian metric $h$ on a line bundle $L\to M$ of degree is $d$ is {\bf admissible}
 in the sense of \cite{Fg}
if 
\begin{equation} \label{eqn:admissible}
\Ric(h)=-(2\pi id)\mu
\end{equation}
The Arakelov metric on $M$, considered as a hermitian metric on the anti-canonical bundle $K^\ast$, is admissible:
\begin{equation} \label{eqn:ricci}
\Ric(\rho_{\smallAr})=4\pi i(g-1)\mu
\end{equation}
In terms of the Hermitian-Einstein tensor and the scalar curvature, 
\eqref{eqn:admissible} and \eqref{eqn:ricci} become
\begin{align}
\begin{split} \label{eqn:curvatures}
dA\, \Omega_{L,h} &= (2\pi d)\mu  \\
dA\, R_{\rho_{\smallAr}}&= -8\pi(g-1)\mu  
\end{split}
\end{align}
For more details we refer to the papers cited above.

We now return to the situation in the previous section.  Let $R_\varepsilon=M\setminus B_\varepsilon$, where $B_\varepsilon$ is the coordinate neighborhood $|z|<\varepsilon$ centered at a point $p$.  Let $L\to M$ be a holomorphic line bundle with admissible metric $h$, and let $L(p)=L\otimes \Ocal(p)$.  Choosing an admissible metric on $\Ocal(p)$ gives an admissible metric on $L(p)$. 
Let $\hat\omega_0$ be a global holomorphic section of $L(p)$ that is nonvanishing at $p$, and let $\onep$ be  a global holomorphic section of $\Ocal(p)$ vanishing at $p$.  Using the framings given by $\hat\omega_0\otimes \onep^{-1}$ and $\hat\omega_0$, respectively, then 
 on $R_\varepsilon$,  $L$ and $L(p)$ are naturally isomorphic as framed bundles, and their hermitian metrics are conformal with factor $f(z)=-\log G(z,p)$. With this understood, we have the following simple computation.

\begin{lemma} \label{lem:liouville}
Let $S_\varepsilon(f)=S(0,f)$ denote the Liouville action \eqref{eqn:liouville} on $R_\varepsilon$.  Then $S_\varepsilon(f)\to 0$ as $\varepsilon\to 0$.
\end{lemma}

\begin{proof}
Note that the local expression for the metric in the framing on $L(p)$ is continuous as $\varepsilon\to 0$.  Hence, if we let $\hat h$ denote the metric on $L(p)$ and $h$ that on $L$, then by \eqref{eqn:liouville},
$$
S_\varepsilon(f)=-\frac{1}{\pi}\int_{M_\varepsilon} dA_{\rho}|\nabla f|^2
-\frac{1}{2\pi}\int_{M_\varepsilon} dA_{\rho}(4\Omega_{L(p),\hat h}+R_\rho)f
+\frac{1}{\pi}\int_{\partial M_\varepsilon} ds_\rho (2\nu_{L(p),\hat h}-\kappa_\rho)f
$$
Using \eqref{eqn:normalization}, \eqref{eqn:curvatures},  and the remark above,
\begin{align*}
S_\varepsilon(f) &\simeq -\frac{1}{\pi}\int_{M_\varepsilon} dA_{\rho}|\nabla f|^2
-\frac{1}{\pi}\int_{\partial M_\varepsilon} ds_\rho \, \kappa_\rho f \\
&= \frac{1}{\pi}\int_{M_\varepsilon} dA_{\rho}\, f\Delta f 
-\frac{1}{\pi}\int_{\partial M_\varepsilon} dA_{\rho}\, f\partial_n f
-\frac{1}{\pi}\int_{\partial M_\varepsilon} ds_\rho \, \kappa_\rho f \\
&\simeq
-\frac{1}{\pi}\int_{\partial M_\varepsilon} dA_{\rho}\, f\partial_n f
-\frac{1}{\pi}\int_{\partial M_\varepsilon} ds_\rho \, \kappa_\rho f
\end{align*}
which vanishes as $\varepsilon\to 0$.
\end{proof}

\subsection{Proof of  Theorem \ref{thm:insertion}} \label{sec:insertion}
Let $L\to M$ with  $\deg L=d$ and $h^1(L)=0$, and set $L(p)=L\otimes\Ocal(p)$.
  Set $N=h^0(L)=d-g+1$ and $m=2N+1$.
 Let  $\{\omega_i\}_{i=1}^{N}$ be a fixed basis for $H^0(M,L)$, 
and set $\hat\omega_i=\omega_i\otimes {\onep}$.
We assume that the framings  $\hat\omega_0\otimes {\one}^{-1}_p$ and
 $\hat\omega_0$ are generic and good in the sense of Definitions \ref{def:generic_framing} and \ref{def:good_framing}. 
 We will need  technical results on degenerations of sections. 
 The proofs of the following two lemmas are straightforward and will be omitted.
 
 \begin{lemma} \label{lem:degeneration_cokernel}
 With the assumption $h^1(L)=0$, $\ker P_L^\dagger$ $($and therefore also $\ker P_{L(p)}^\dagger$$)$ vanishes on $R_\varepsilon$ for $\varepsilon>0$ sufficiently small.
 \end{lemma}
 
 \begin{lemma} \label{lem:degeneration_kernels1}
Let $\nu_i$ be the order of vanishing of $\omega_i$  at $p$.
 Then for any  sequence $\varepsilon_k\to 0$ there is a subsequence $($also denoted $\{\varepsilon_k\}$ $)$ and
a collection $\{\omega_{i,\varepsilon_k}\}_{i=1}^{m}$, 
$\omega_{m,\varepsilon_k}=\hat\omega_0\otimes {\one}^{-1}_p$ 
for all $k$, satisfying the following. 
\begin{itemize}
\item 
The set $\{\omega_{i,\varepsilon_k}\}_{i=1}^{m}$
is a real basis for the 
subspace of $\ker\dbar_L$ on $R_{\varepsilon_k}$
 with 
 $\omega_{i,\varepsilon_k}=f_{i,\varepsilon_k}\hat\omega_0\otimes
 {\one}_p^{-1}$ near $p$ satisfying
  $\imag (f_{i,\varepsilon_k})\bigr|_{|z|=\varepsilon_k}=0$ (cf.\ Remark \ref{rem:kernels}).  
\item 
 For each $1\leq j\leq N$,
\begin{align*} 
 \sup_{z\in R_{\varepsilon_k}}\varepsilon_k^{-\nu_i}\left|
\omega_{2j-1,\varepsilon_k}(z)-\omega_j(z)\right|&\lra 0
\\
 \sup_{z\in R_{\varepsilon_k}}
\varepsilon_k^{-\nu_i}\left|\omega_{2j,\varepsilon_k}(z)-i\omega_j(z)\right|&\lra 0
\end{align*} 
 as $k\to \infty$.
\end{itemize}
 \end{lemma}
 
Set $\Phi_m = \jmath(\hat\omega_0\otimes\one_p^{-1})$, 
$\widehat \Phi_m = \jmath(\hat\omega_0)$, and
$\widehat \Phi_{m+1} = \jmath(i\hat\omega_0)$.
For $1\leq i\leq m$,  
set $\Phi_{i,\varepsilon_k}=\jmath(\omega_{i,\varepsilon_k})$,  
 $\widehat\Phi_{i,\varepsilon_k}=\jmath(\omega_{i,\varepsilon_k}\otimes\onep)$, and
for $1\leq j\leq N$,  set 
$$
 \begin{array} {lclclcl}
  \Phi_{2j} & = &\jmath(i\omega_j) && \widehat\Phi_{2j} & = &
 \jmath(i\hat\omega_{2j}) \\
  \Phi_{2j-1} & = &\jmath(\omega_j) && \widehat\Phi_{2j-1} & = &
 \jmath(\hat\omega_{2j}) 
   \end{array}
$$
We will need the following asymptotics for the sections chosen as above.
\begin{lemma} \label{lem:boundary_det}  
Assume without loss of generality that the metric on $M$ is locally euclidean on a neighborhood of $p$. Then
as $k\to \infty$,
\begin{align*}
{\det}^\ast(\widehat\Phi_i'',\widehat\Phi_j'')_{{\partial R}_{\varepsilon_k}} &\simeq
{\det}(\Phi_i'',\Phi_j'')_{{\partial R}_{\varepsilon_k}}
(2\pi\varepsilon_k\Vert\widehat\Phi_m(p)\Vert^2)
 (\rho_{\smallAr}(p)\varepsilon_k^2)^{m-1} \\
\det(\widehat\Phi^{\smallA}_{i,\varepsilon_k},\widehat\Phi^{\smallA}_{j,\varepsilon_k})_{{\partial R}_{\varepsilon_k}} &\simeq
\det(\Phi^{\smallA}_{i,\varepsilon_k},\Phi^{\smallA}_{j,\varepsilon_k})_{{\partial R}_{\varepsilon_k}} (\rho_{\smallAr}(p)\varepsilon_k^2)^m  
\end{align*}
(where here because of the choice of  indexing, $\det^\ast$ denotes minus the determinant of the $m\times m$ minor, unlike in Theorem \ref{thm:bfk_adapted}).
\end{lemma}

\begin{proof}  
It suffices to prove an estimate for $\det\langle\omega_{2i,\varepsilon_k},\omega_{2j,\varepsilon_k}\rangle_{{\partial R}_{\varepsilon_k}}$.  
Write the expansion
$$
f_{2i,\varepsilon_k}(z)=\sum_{n\in\Z} a^{(i)}_{n,\varepsilon_k} z^n
$$
Note that the condition $\imag(f_{2i,\varepsilon})\bigr|_{|z|=\varepsilon_k}=0$ 
 implies $a^{(i)}_{-n,\varepsilon_k}=\bar a^{(i)}_{n,\varepsilon_k}
\varepsilon_k^{2n}$, for $n\geq 0$.
By Lemma \ref{lem:degeneration_kernels1}, we have $|a^{(i)}_{\nu_i,\varepsilon_k}|\neq 0$ and
$|\varepsilon_k^{n-\nu_i} a^{(i)}_{n,\varepsilon_k}|\to 0$, for $n\neq \nu_i$.
Hence,
\begin{align*}
\langle\omega_{2i,\varepsilon_k},\omega_{2j,\varepsilon_k}
\rangle_{{\partial R}_{\varepsilon_k}}
&\simeq \sum_{\ell,n}\int_0^{2\pi}\, a^{(i)}_{\ell,\varepsilon_k}
 \bar a^{(j)}_{n,\varepsilon_k} z^\ell\bar z^n \varepsilon_k\, d\theta
 \times \frac{\Vert\hat\omega_0(p)\Vert^2}{(\rho_{\smallAr}(p)\varepsilon_k^2)}
\\
&\simeq 2\pi\varepsilon_k^{-1} \sum_{n\in \Z} a^{(i)}_{n,\varepsilon_k}
 \bar a^{(j)}_{n,\varepsilon_k} \varepsilon_k^{2n} 
 \times \frac{\Vert\hat\omega_0(p)\Vert^2}{\rho_{\smallAr}(p)}
\\
\varepsilon_k^{-\nu_i-\nu_j+1}\langle\omega_{2i,\varepsilon_k},\omega_{2j,
\varepsilon_k}\rangle_{{\partial R}_{\varepsilon_k}}
&\simeq 2\pi \sum_{n\in \Z} \varepsilon_k^{n-\nu_i} 
 a^{(i)}_{n,\varepsilon_k} \varepsilon_k^{n-\nu_j} \bar a^{(j)}_{n,\varepsilon_k}  
 \times \frac{\Vert\hat\omega_0(p)\Vert^2}{\rho_{\smallAr}(p)}
\\
&\simeq 2\pi |a^{(i)}_{\nu_i,0}|^2\delta_{ij}
 \times \frac{\Vert\hat\omega_0(p)\Vert^2}{\rho_{\smallAr}(p)}
+o(1)
\end{align*}
Similarly, 
$$
\varepsilon_k^{-\nu_i-\nu_j-1}\langle\hat\omega_{2i,\varepsilon_k},\hat\omega_{2j,
\varepsilon_k}\rangle_{{\partial R}_{\varepsilon_k}}
\simeq 2\pi |a^{(i)}_{\nu_i,0}|^2\delta_{ij}
 \times \Vert\hat\omega_0(p)\Vert^2
+o(1)
$$
The second estimate in the lemma follows from this.
The proof of the first estimate is similar.
\end{proof}

We now turn to the proof of Theorem \ref{thm:insertion}. 
On the one hand, for $L$ we may apply \eqref{eqn:sep_curve} to get
\begin{equation}
\left[\frac{{\Det}^\ast\dl}{\det(\Phi_i, \Phi_j)}\right]_M=c_Q\, 
\left[\frac{{\Det}^\ast\dla}{\det(\phia_{i,\varepsilon_k},\phia_{j,\varepsilon_k})}\right]_{R_{\varepsilon_k}}
\biggl[ {\Det}^\ast D^{\smallA}_K\biggr]_{B_{\varepsilon_k}} 
\frac{{\det}(\phia_{i,\varepsilon_k},\phia_{j,\varepsilon_k})_\Gamma }{{\det}(\Phi_i'',\Phi_j'')_\Gamma}\, {\Det}_Q^\ast \N_\Gamma
\end{equation}
On the other hand, from Theorem \ref{thm:bfk_adapted} applied to $L(p)$, we obtain
\begin{equation}
\left[\frac{{\Det}^\ast D_{L(p)}}{\det(\widehat\Phi_i, \widehat\Phi_j)}\right]_M=c_Q\, 
\left[\frac{{\Det}^\ast D^{\smallA}_{L(p)}}{\det(\widehat\Phi^{\smallA}_{i,\varepsilon_k},\widehat\Phi^{\smallA}_{j,\varepsilon_k})}\right]_{R_{\varepsilon_k}}
\left[ \frac{ {\Det}^\ast D^{\smallA}_\Ocal}{ \Vert \widehat\Phi_m\Vert^2       }\right]_{B_{\varepsilon_k}} 
\frac{{\det}(\widehat\Phi^{\smallA}_{i,\varepsilon_k},\widehat\Phi^{\smallA}_{j,\varepsilon_k})_\Gamma }{{\det}^\ast(\widehat\Phi_i'',\widehat\Phi_j'')_\Gamma}\, {\Det}_Q^\ast \widehat\N_\Gamma
\end{equation}
Now on $R_{\varepsilon_k}$, the framed bundles $L$ and $L(p)$ are isomorphic, and by Lemma \ref{lem:liouville}
$$
\left[\frac{{\Det}^\ast\dla}{\det(\phia_{i,\varepsilon_k},\phia_{j,\varepsilon_k})}\right]_{R_{\varepsilon_k}}\simeq
\left[\frac{{\Det}^\ast D^{\smallA}_{L(p)}}{\det(\widehat\Phi^{\smallA}_{i,\varepsilon_k},\widehat\Phi^{\smallA}_{j,\varepsilon_k})}\right]_{R_{\varepsilon_k}}
$$
 By Proposition \ref{prop:serre},  ${\Det}^\ast D^{\smallA}_K={\Det}^\ast D^{\smallA}_\Ocal$ on the disk.  Applying Propositions \ref{prop:dn_pinch_one} and \ref{prop:dn_pinch_two} (using Lemma \ref{lem:degeneration_cokernel} and noting that $h^0(L(p))=h^0(L)+1$), 
\begin{align*}
c_Q{\Det}_Q^\ast \N_\Gamma &\simeq (2/\varepsilon_k)2^{-4h^0(L(p))+2} \\
c_Q{\Det}_Q^\ast \widehat\N_\Gamma &\simeq  2^{-4h^0(L(p))+2} 
\end{align*}
Hence,

$$
\left[\frac{{\Det}^\ast D_{L(p)}}{\det(\widehat\Phi_i, \widehat\Phi_j)}\right]_M
\simeq
 \frac{(\varepsilon_k/2)}{\Vert \widehat\Phi_m\Vert^2_{B_{\varepsilon_k}} }
\frac{{\det}(\widehat\Phi^{\smallA}_{i,\varepsilon_k},\widehat\Phi^{\smallA}_{j,\varepsilon_k})_\Gamma }{{\det}^\ast(\widehat\Phi_i'',\widehat\Phi_j'')_\Gamma}
\left[\frac{{\det}(\phia_{i,\varepsilon_k},\phia_{j,\varepsilon_k})_\Gamma }{{\det}(\Phi_i'',\Phi_j'')_\Gamma}\right]^{-1}
\left[\frac{{\Det}^\ast\dl}{\det(\Phi_i, \Phi_j)}\right]_M
$$
Finally, 
$$
\Vert \widehat\Phi_m\Vert^2_{B_{\varepsilon_k}}\simeq \pi\varepsilon_k^2\, 
\Vert\Phi_m(p)\Vert^2\, \rho_{\smallAr}(p) 
$$
Combining this with
Lemma \ref{lem:boundary_det} and letting $k\to \infty$, we have

$$
4\pi^2\Vert \widehat\Phi_0\Vert^4(p)\left[\frac{{\Det}^\ast D_{L(p)}}{\det(\widehat\Phi_i, \widehat\Phi_j)}\right]_M
=
\left[\frac{{\Det}^\ast\dl}{\det(\Phi_i, \Phi_j)}\right]_M
$$
The result now follows from Lemma \ref{lem:real_complex}.

\end{document}